\documentclass[12pt,fleqn]{article}

\usepackage{amsmath,amssymb,amsthm}
\usepackage{geometry} 
\usepackage[pdfpagemode=UseNone,pdfstartview=FitH]{hyperref}
\usepackage{enumitem}

\usepackage{esint}

\newcommand{\bdry}[1]{\partial #1}

\newcommand{\A}{{\cal A}}
\newcommand{\D}{{\cal D}}
\newcommand{\F}{{\cal F}}

\newcommand{\dist}[2]{\text{dist}\, (#1,#2)}

\newcommand{\incl}{\hookrightarrow}

\newcommand{\M}{{\cal M}}
\newcommand{\N}{\mathbb N}
\newcommand{\norm}[2][]{\left\|#2\right\|_{#1}}

\newcommand{\pnorm}[2][]{\if #1'' \left|#2\right|_p \else \left|#2\right|_{#1} \fi}
\newcommand{\R}{\mathbb R}
\newcommand{\RP}{\R \text{P}}
\newcommand{\seq}[1]{\left(#1\right)}
\newcommand{\set}[1]{\left\{#1\right\}}

\newcommand{\w}[1]{\widetilde{#1}}
\newcommand{\weak}{\rightharpoonup}
\newcommand{\Z}{\mathbb Z}

\newcommand{\om}{\int_{\mathbb{R}^N}}

\newcommand{\dev}[1]{\int_{\mathbb{R}^N}|\nabla#1|^2dx}
\newcommand{\cri}[1]{\int_{\mathbb{R}^N}|#1|^{2^*}dx}

\newcommand{\s}[1]{\left\{#1\right\}_{n\in\mathbb{N}}}
\newcommand{\x}{\mathcal{X}(\mathbb{R}^N)}
\newcommand{\X}{\mathcal{X}_{\text{rad}}(\R^N)}

\newcommand{\Q}{Q_{\gamma}(\R^N)}
\newcommand{\normq}[1]{\|#1\|_{Q_{\gamma}}}
\newcommand{\En}{\mathcal{E}}        
\newcommand{\Bn}{\mathcal{B}}        
\newcommand{\Hn}{\mathcal{H}}        
\newcommand{\qrad}{q_{\text{rad}}^\alpha}

\DeclareMathOperator{\divg}{div}

\newenvironment{enumroman}{\begin{enumerate}

}{\end{enumerate}}

\newtheorem{corollary}{Corollary}[section]
\newtheorem{lemma}[corollary]{Lemma}
\newtheorem{proposition}[corollary]{Proposition}
\newtheorem{theorem}[corollary]{Theorem}

\theoremstyle{definition}
\newtheorem{definition}[corollary]{Definition}

\theoremstyle{remark}
\newtheorem{example}[corollary]{Example}
\newtheorem{remark}[corollary]{Remark}

\numberwithin{equation}{section}

\title{\bf Solutions to critical equations with a superposition of nonlocal Hartree--type nonlinearities }

\author{
\bf Artur Jorge Marinho\\
Department of Mathematics\\
Florida Institute of Technology\\
150 W University Blvd, Melbourne, FL 32901-6975, USA\\
\em amarinho2024@my.fit.edu\\
[\medskipamount]
\bf Kanishka Perera\\
Department of Mathematics\\
Florida Institute of Technology\\
150 W University Blvd, Melbourne, FL 32901-6975, USA\\
\em kperera@fit.edu\\
}
\date{}
\begin{document}
\maketitle

\begin{abstract}
    We study a class of nonlinear nonlocal elliptic equations in $\R^N$ involving superpositions of Hartree-type nonlinearities. Motivated by the Schr\"odinger--Poisson--Slater system, these equations arise as natural generalizations of problems with a single nonlocal interaction term. More precisely, we consider equations driven by a family of Riesz potentials weighted by a positive Borel measure, which gives rise to a superposed nonlocal operator.
To treat this problem variationally, we introduce suitable functional settings, namely the superposed Coulomb space and the associated superposed Coulomb--Sobolev space, and study their main properties; in particular, we show that these spaces are uniformly convex for every admissible choice of the powers. Combining variational methods with a recently developed scaling-based critical point theory, we prove existence and multiplicity results for radial solutions in every dimension $N\geq3$. We also investigate a Brezis--Nirenberg-type problem and obtain multiplicity results near eigenvalues of an associated nonlinear eigenvalue problem.
Our results extend previous works on single Hartree-type equations and provide a unified framework for treating superpositions of nonlocal interactions of Hartree type.
\end{abstract}

\newpage

{\small \tableofcontents}

\newpage

\section{Introduction}

Schr\"odinger--Poisson--Slater equations of the form
\begin{eqnarray}\label{intro1}
    -\Delta u+V(x)u+\seq{\frac{1}{4\pi|x|}\star u^2}u=f(x,u),
    \quad\text{in }\,\R^3,
\end{eqnarray}
have been extensively studied in recent years. They arise as stationary versions of nonlinear Schr\"odinger equations of the type
\begin{eqnarray}\label{intro2}
    i\partial_t\Psi=-\Delta\Psi+V(x)\Psi
    +\seq{\frac{1}{4\pi|x|}\star|\Psi|^2}\Psi
    -f(x,\Psi),
    \quad\text{in }\, \R^3\times\R.
\end{eqnarray}
Equation \eqref{intro2} can be viewed as an approximation of the Hartree--Fock model, which appears in the physical description of many-body systems of electrons. In this setting, the term $u^2$ in \eqref{intro1} represents the electronic density of the system, the nonlocal convolution term describes the Coulomb interaction among the electrons, and $V$ denotes an external potential.

For a concise introduction to the mathematical aspects of the Hartree--Fock model, see \cite{MR2149087}. We also refer the reader to the works of Lieb \cite{MR629207,MR641371}, Le Bris and Lions \cite{MR2149087}, Lu and Otto \cite{MR3251907}, Frank et al.\ \cite{MR3762278}, and the references therein. The nonlinear term $f(x,u)=|u|^{2/3}u$ was introduced by Slater \cite{PhysRev.81.385} as an approximation of the exchange potential in the Hartree--Fock model; see also \cite{doi:10.1142/S0218202503002969, MAUSER2001759, doi:10.1142/S0218202599000439}. The zero-mass variant of equation \eqref{intro1},
\begin{equation}\label{intro3}
    -\Delta u+\left(\frac{1}{4\pi|x|}\star u^2\right)u=f(x,u),
    \quad\text{in }\R^3,
\end{equation}
has also been studied in recent years. When $f(x,u)=|u|^{q-2}u$ with $q\in(2,3)$, equation \eqref{intro3} can be viewed as a formal limiting problem associated with
\[
-\Delta u+u+\lambda\left(\frac{1}{4\pi|x|}\star u^2\right)u
=|u|^{q-2}u,\quad\text{in }\R^3,
\]
as $\lambda\to0$ \cite{MR2679375}.

The zero-mass problem with power-type nonlinearity $f(x,u)=|u|^{q-2}u$ was studied by Ruiz \cite{MR2679375}, who developed in detail the variational framework associated with the equation. In particular, he proved that for every $q\in(18/7,3)$, equation \eqref{intro3} admits a nontrivial radial solution. Ianni and Ruiz \cite{MR2902293} later showed the existence of a positive solution for any $q\in(3,6)$.

Motivated by the scaling properties of problem \eqref{intro3}, Mercuri and Perera \cite{MePe1} introduced a new method in critical point theory, which they used to prove new existence and multiplicity results for \eqref{intro3}. In particular, they generalized the classical saddle point theorem and linking theorem of Rabinowitz to this scaled variational structure; see \cite[Theorems 2.23 and 2.24]{MePe1} and also \cite[Theorems 2.11 and 2.12]{MR1400007}. A key ingredient in this method is the analysis of an unbounded sequence
\[
0<\lambda_1\leq \lambda_2\leq \cdots \leq \lambda_k\leq \cdots \to +\infty
\]
of eigenvalues for the associated nonlinear eigenvalue problem
\[
-\Delta u+\left(\frac{1}{4\pi|x|}\star u^2\right)u
=\lambda |u|u,\quad\text{in }\R^3,
\]
which was constructed by means of the $\Z_2$-cohomological index of Fadell and Rabinowitz \cite{MR0478189}; see \cite[Theorem 2.10]{MePe1}.

This framework was further generalized by Mercuri et al. \cite{MR3568051}, who studied equations of the form
\begin{equation}\label{intro4}
    -\Delta u+\left(I_\alpha\star|u|^p\right)|u|^{p-2}u
    =|u|^{q-2}u,\quad\text{in }\,\R^N,
\end{equation}
where $N\geq 1$, $p,q>1$, and $I_\alpha=A_\alpha |x|^{\alpha-N}$ is the Riesz potential of order $\alpha\in(0,N)$. They introduced and studied the {\sl Coulomb--Sobolev} space, namely the natural energy space associated with \eqref{intro4}, and proved the existence of nonnegative solutions to \eqref{intro4} for a range of values of $q>1$; see \cite[Theorems 3 and 5]{MR3568051}.

Using the variational structure developed in \cite{MR3568051}, Marinho, Mercuri and Perera \cite{MaMePe1} recently applied the scaled variational method of Mercuri and Perera to equations of the form
\begin{equation}\label{intro5}
    -\Delta u+\left(I_\alpha\star|u|^{p}\right)|u|^{p-2}u
    =f(|x|,u),\quad\text{in }\, \R^N,
\end{equation}
where the Riesz potential has order $\alpha\in(1,N)$ and $f$ is a Carath\'eodory function satisfying suitable growth assumptions. The nonlocal term
\[
\left(I_\alpha\star|u|^{p}\right)|u|^{p-2}u
\]
is usually called a {\sl Hartree-type} nonlinearity. Equation \eqref{intro5} exhibits a rich variational structure due to its dependence on the parameters $N$, $p$, and $\alpha$. In particular, new phenomena arise when $p$ is sufficiently large, making the problem more difficult to treat. We refer the reader to \cite{MaMePe1} for a more detailed discussion of these issues.

Due to the greater flexibility in the parameters appearing in equation \eqref{intro5}, it is natural to ask whether one can treat problems involving more than one Hartree-type nonlinearity. For instance, one may consider equations of the form
\[
-\Delta u+\left(I_\alpha\star|u|^{p}\right)|u|^{p-2}u
+\left(I_\beta\star|u|^{q}\right)|u|^{q-2}u
=f(|x|,u),\quad\text{in }\,\R^N,
\]
where $\alpha,\beta\in(0,N)$ and $p,q>1$. More generally, one may ask what new difficulties arise in the study of problems involving a countable superposition of Hartree-type nonlinearities, such as
\[
-\Delta u+\sum_{i=1}^\infty c_i
\left(I_{\beta_i}\star|u|^{p_i}\right)|u|^{p_i-2}u
=f(|x|,u),\quad\text{in }\,\R^N,
\]
where $c_i>0$, $p_i>1$, and $\beta_i\in(0,N)$ for all $i\geq1$. At an even higher level of generality, one may consider problems involving possibly uncountable superpositions of Hartree-type nonlinearities. The present paper gives a partial answer to these questions by developing a framework suitable for treating some types of superpositions.

 Let $0<a<b<N$, and let $\mu$ be a finite positive Borel measure on $(0,N)$, not identically zero, such that
\[
\mu\big((0,N)\setminus [a,b]\big)=0.
\]
Let also $\gamma>\frac{N-2}{2}$ be fixed. In this paper, we prove existence and multiplicity results for the nonlinear and nonlocal problem
\begin{eqnarray}\label{superposition problem}
    -\Delta u+\int_{(0,N)}
    \left(I_{\beta}\star|u|^{q(\beta)}\right)
    |u|^{q(\beta)-2}u\,d\mu(\beta)
    =
    \lambda |u|^{q_\gamma-2}u
    +\eta |u|^{r-2}u
    +|u|^{2^*-2}u
\end{eqnarray}
in $\R^N$, where $\lambda\in\R$, $\eta\geq0$, and $r\in[q_\gamma,2^*)$. Here
\[
q(\beta)=1+\dfrac{\beta+2}{2\gamma},
\qquad
q_\gamma=2+\frac{2}{\gamma},
\qquad
2^*=\dfrac{2N}{N-2},
\]
with $\beta\in(0,N)$ and $N\geq3$. Note that $q(\beta)>1$ for every $\beta\in(0,N)$ and $\gamma>0$, and that $\gamma>\frac{N-2}{2}$ is equivalent to $q_\gamma<2^*$.

The function $I_\beta:\R^N\setminus\{0\}\to\R$ denotes the Riesz potential of order $\beta\in(0,N)$, defined by
\[
I_\beta(x)=\dfrac{A_\beta}{|x|^{N-\beta}},
\qquad
A_\beta=\dfrac{\Gamma\left(\frac{N-\beta}{2}\right)}
{\Gamma\left(\frac{\beta}{2}\right)\pi^{N/2}2^\beta}.
\]
The choice of the constant $A_\beta$ ensures that the kernels $I_\beta$ satisfy the semigroup property
\[
I_{\beta_1}\star I_{\beta_2}=I_{\beta_1+\beta_2}
\]
whenever $\beta_1,\beta_2\in(0,N)$ and $\beta_1+\beta_2<N$; see, for instance, \cite[pp. 73--74]{du1970introduction}. Consequently, the energy associated with the Hartree-type term in \eqref{superposition problem} can be treated as an $L^2$-type quantity, since
\[
\iint_{\R^N\times\R^N}
\dfrac{A_\beta |u(x)|^{q(\beta)}|u(y)|^{q(\beta)}}
{|x-y|^{N-\beta}}\,dx\,dy
=
\int_{\R^N}
\left|I_{\beta/2}\star |u|^{q(\beta)}\right|^2\,dx
\]
for every $\beta\in(0,N)$. Note also that $\beta\mapsto A_\beta$ is continuous and positive on $(0,N)$, and hence bounded from above and from below by positive constants on $[a,b]$; this fact will be used repeatedly to obtain estimates that are uniform with respect to $\beta\in[a,b]$.

Problem \eqref{superposition problem} contains several previously studied models as particular cases. For instance, when the measure $\mu$ is a Dirac mass, the nonlocal term reduces to a single Hartree-type nonlinearity. In particular, suitable choices of the parameters recover equations of the form \eqref{intro3} and \eqref{intro5}, after choosing the corresponding right-hand side. For example, $N=3$, $\mu=\delta_2$ (note that $I_2(x)=\frac{1}{4\pi|x|}$ in $\R^3$) and $\gamma=2$ give $q(2)=2$ and $q_\gamma=3$, which is the setting of \cite{MePe1}.

The choice of the powers $q(\beta)$ is dictated by the scaling properties of the problem and allows us to employ the scaling-based approach in critical point theory recently developed by Mercuri and Perera in \cite{MePe1}. A central ingredient is the analysis of the nonlinear eigenvalue-type problem associated with \eqref{superposition problem}, namely
\begin{eqnarray}\label{e1}
    -\Delta u+\int_{(0,N)}
    \left(I_{\beta}\star|u|^{q(\beta)}\right)
    |u|^{q(\beta)-2}u\,d\mu(\beta)
    =
    \lambda |u|^{q_\gamma-2}u,
    \quad\text{in }\R^N.
\end{eqnarray}
This equation can be interpreted as a nonlinear eigenvalue problem through its scaling invariance. It therefore belongs to the broader class of scaled eigenvalue problems studied in \cite{MePe1}, where a suitable unbounded sequence of eigenvalues $\lambda_k\nearrow+\infty$ was constructed by means of the $\Z_2$-cohomological index of Fadell and Rabinowitz \cite{MR0478189}.

A related unbounded sequence of eigenvalues can also be defined using the Krasnosel'skii genus. However, the $\Z_2$-cohomological index satisfies the piercing property, whereas the genus does not. This property plays a crucial role in the construction of suitable minimax levels, which in turn yield critical values for the functional associated with problem \eqref{superposition problem}.

Since the constant $A_\beta$ blows up as $\beta$ approaches the boundary of the interval $(0,N)$, it is necessary to require the measure $\mu$ to vanish near this boundary; this is the reason for the assumption $\mu((0,N)\setminus[a,b])=0$ with $[a,b]\subset(0,N)$.

Since problem \eqref{superposition problem} is treated by variational methods, a natural question is to identify the appropriate energy space for this setting. The energy functional associated with \eqref{superposition problem} is given by
\begin{eqnarray*}
    \Phi(u)
    &=&
    \dfrac{1}{2}\int_{\R^N}|\nabla u|^2\,dx
    +\int_{(0,N)}
    \dfrac{1}{2q(\beta)}
    \int_{\R^N}
    \left|I_{\beta/2}\star |u|^{q(\beta)}\right|^2\,dx\,d\mu(\beta)
    \\[0.2cm]
    &&-
    \dfrac{\lambda}{q_\gamma}\int_{\R^N}|u|^{q_{\gamma}}\,dx
    -\dfrac{\eta}{r}\int_{\R^N}|u|^r\,dx
    -\dfrac{1}{2^*}\int_{\R^N}|u|^{2^*}\,dx .
\end{eqnarray*}
To define the underlying space that best fits this framework, we proceed as follows. We first introduce the {\sl superposed Coulomb space} $Q_\gamma(\R^N)$ as the set of all measurable functions $u:\R^N\to\R$ for which there exists $\tau=\tau(u)>0$ such that
\[
\rho(u/\tau):=\int_{(0,N)}\int_{\R^N}\left|I_{\beta/2}\star\left|\frac{u}{\tau}\right|^{q(\beta)}\right|^2dx\,d\mu(\beta)<+\infty.
\]
On $Q_\gamma(\R^N)$, we define a Luxemburg-type norm by
\[
\|u\|_{Q_\gamma}:=
\inf\set{\tau>0:\rho(u/\tau)\leq1}.
\]
We then study the properties of this normed space.

After that, we define the {\sl superposed Coulomb--Sobolev space} $\x$ by
\[
\x:=\mathcal{D}^{1,2}(\R^N)\cap Q_{\gamma}(\R^N),
\]
endowed with the norm
\[
\|u\|_{\x}:=\|\nabla u\|_{L^2(\R^N)}+\|u\|_{Q_\gamma}.
\]
The properties of this space are then analyzed.

In the particular case where $N=3$ and $\mu=\delta_2$, this type of space was studied in \cite{MR2679375}. In the more general case where $N\geq3$ and $\mu=\delta_\alpha$ for some $\alpha\in(0,N)$, the corresponding Coulomb--Sobolev space was studied in \cite{MR3568051}. For instance, it was shown in \cite{MR3568051} that $\x$ is a uniformly convex Banach space. Moreover, the authors identified ranges of powers $r\geq1$ for which the embedding
\[
\x\hookrightarrow L^r(\R^N)
\]
is continuous, as well as ranges for which this embedding is compact; see \cite[Theorems 1, 4 and 5]{MR3568051}. In Section \ref{uniform-convexity} we show that the superposed space $\x$ is uniformly convex for every $\gamma>0$, that is, whenever all the powers $q(\beta)$ are larger than one; this extends \cite[Proposition 2.8]{MR3568051} to superpositions, and it is what allows us to work in every dimension $N\geq3$ (note that for $N\geq4$ our assumptions force $q(\beta)<2$; see Remark \ref{admissible}).

For a general positive Borel measure $\mu$, it is reasonable to expect that the embedding properties of $\x$ depend on the behavior of $\mu$. We say that $\alpha\in(0,N)$ has the positive neighborhood property with respect to $\mu$ if
\[
\mu(I)>0
\]
for every open interval $I\subset(0,N)$ containing $\alpha$. We denote by $\mathcal{P}_\mu$ the set of all points $\alpha\in (0,N)$ which have the positive neighborhood property with respect to $\mu$ (that is, $\mathcal P_\mu$ is the support of $\mu$). We will see that this set will play an important role in the structure of the underlying space. More precisely, following the ideas developed in \cite{MR3568051}, we obtain the following result, where $\X$ denotes the subspace of radial functions of $\x$.

\begin{theorem}\label{th0}
      Let $N\geq3$, $\gamma>0$, and assume that $\mathcal{P}_\mu\cap(1,N)\neq\emptyset$. Let $\alpha\in \mathcal{P}_\mu\cap(1,N)$. If
    \begin{enumroman}
        \item either $q(\alpha)<\dfrac{N+\alpha}{N-2}$ and $q_{\text{rad}}^\alpha<r\leq2^*$,
        \item  or $q(\alpha)>\dfrac{N+\alpha}{N-2}$, and $2^*\leq r<q_{\text{rad}}^\alpha$
        \end{enumroman}
       where $q_{\text{rad}}^\alpha=\frac{2(2q(\alpha)(N-1)+N-\alpha)}{3N+\alpha-4}$, then  $\X\hookrightarrow L^r(\R^N)$ continuously. Moreover, if all inequalities are strict, then the embedding is compact.
\end{theorem}
\begin{remark}\label{rem-qgamma}
   A direct computation shows that $q_\gamma=2\frac{2q(\alpha)+\alpha}{2+\alpha}$ for every $\alpha\in(0,N)$, and that $q_\gamma<2^*$ (that is, $\gamma>\frac{N-2}{2}$) is equivalent to $q(\alpha)<\frac{N+\alpha}{N-2}$. Moreover, the affine functions $p\mapsto 2\frac{2p+\alpha}{2+\alpha}$ and $p\mapsto \frac{2(2p(N-1)+N-\alpha)}{3N+\alpha-4}$ coincide (with value $2^*$) at $p=\frac{N+\alpha}{N-2}$, and the second one has the larger slope when $\alpha>1$. Consequently, if $\gamma>\frac{N-2}{2}$ and $\alpha\in(1,N)$, then only case (i) of Theorem \ref{th0} occurs and
$q_\gamma\in(q_{\mathrm{rad}}^\alpha,2^*)$. Hence, by Theorem~\ref{th0}, the embedding
\[
\X\hookrightarrow L^{q_\gamma}(\R^N)
\]
is compact.

The assumption $\alpha>1$ cannot be dropped. Already when $\mu=\delta_\alpha$ is a Dirac mass, the embedding of the radial Coulomb--Sobolev space into $L^{q_\gamma}(\R^N)$, which is continuous by \cite[Theorems 1 and 4]{MR3568051}, fails to be compact if $\alpha<1$ by \cite[Lemma 6.2]{MR3568051}, and if $\alpha=1$ by \cite[Proposition 6.6]{MR3568051}. Since this compactness property plays a crucial role
in the present work, we require the existence of at least one
$\alpha\in(1,N)$ satisfying the positive neighborhood property.
\end{remark}

As pointed out above, compactness is essential in order to employ the machinery developed in \cite{MePe1}. For this reason, we work in the radial subspace $\X\subset\x$ and seek radial solutions of problem \eqref{superposition problem}. Throughout the rest of the introduction we assume the following.

\medskip
\noindent{\bf Standing assumptions.}
\begin{enumerate}[label=$(A_{\arabic*})$]
    \item\label{A1} $N\geq3$, $0<a<b<N$, and $\mu\not\equiv0$ is a finite positive Borel measure on $(0,N)$ with $\mu\big((0,N)\setminus[a,b]\big)=0$;
    \item\label{A2} $\gamma>\dfrac{(N-2)(b+2)}{4}$, or equivalently $q(b)<\dfrac{N}{N-2}$;
    \item\label{A3} $\mathcal{P}_\mu\cap(1,N)\neq\emptyset$.
\end{enumerate}

\begin{remark}[The admissible parameters]\label{admissible}
    Since $b>0$, assumption \ref{A2} implies $\gamma>\frac{N-2}{2}$, that is, $q_\gamma<2^*$. Since $q(\cdot)$ is increasing, \ref{A2} says that $1<q(\beta)<\frac{N}{N-2}=\frac{2^*}{2}$ for all $\beta\in[a,b]$; in particular $q(\beta)<2$ for $\beta\in[a,b]$ when $N\geq4$. Assumption \ref{A3} forces $b>1$. Conversely, for every $N\geq3$ and every measure $\mu$ satisfying \ref{A1} and \ref{A3}, assumption \ref{A2} holds for all sufficiently large $\gamma$. For instance:
    \begin{itemize}
        \item $N=3$ and $\mu=\delta_2$ (the Coulomb kernel $I_2(x)=\frac{1}{4\pi|x|}$): \ref{A2} reads $\gamma>1$; the choice $\gamma=2$ gives $q(2)=2$, $q_\gamma=3$;
        \item $N\geq3$ and $d\mu=\mathbf{1}_{[a,b]}\,d\beta$ with $0<a<b<N$ and $b>1$: \ref{A2} reads $\gamma>\frac{(N-2)(b+2)}{4}$; e.g.\ $N=4$, $[a,b]=[1,2]$ and any $\gamma>2$.
    \end{itemize}
\end{remark}

Our main results can now be stated.

\begin{theorem}\label{th1}
    Assume \ref{A1}--\ref{A3} and let $r\in(q_\gamma,2^*)$. Then, for any $\lambda\in\R$ and $m\geq1$, there exists
    $\eta_m=\eta_m(\lambda)>0$ such that equation
    \eqref{superposition problem} has $m$ distinct pairs of nontrivial radial
    solutions at positive energy levels for all $\eta>\eta_m$.
    In particular, the number of solutions tends to infinity as
    $\eta\to\infty$.
\end{theorem}

\begin{corollary}\label{cor1}
    Assume that $N\geq 3$, $0<a<b<N$ and $\gamma>\frac{(N-2)(b+2)}{4}$, and let $r\in(q_\gamma,2^*)$. Let
    $\{c_i\}_{i\in\N}$ be a sequence of nonnegative numbers
    such that
    \[
    \sum_{i=1}^\infty c_i<\infty,
    \]
    and let $\{\beta_i\}_{i\in\N}\subset [a,b]$. Suppose that for some $j\geq1$, one has $c_j>0$ and $\beta_j\in(1,N)$. Then, for any
    $\lambda\in\R$ and $m\geq1$, there exists
    $\eta_m=\eta_m(\lambda)>0$ such that the equation
    \[
    -\Delta u+
    \sum_{i=1}^\infty c_i
    \left(I_{\beta_i}\star |u|^{q(\beta_i)}\right)
    |u|^{q(\beta_i)-2}u
    =
    \lambda |u|^{q_\gamma-2}u
    +\eta |u|^{r-2}u
    +|u|^{2^*-2}u,
    \quad\text{in }\,\R^N,
    \]
    has $m$ distinct pairs of nontrivial radial solutions at positive energy levels
    for all $\eta>\eta_m$. In particular, the number of solutions tends to
    infinity as $\eta\to\infty$.
\end{corollary}
\begin{proof}
    Apply Theorem \ref{th1} with $\mu=\sum_{i}c_i\delta_{\beta_i}$. Then \ref{A1} and \ref{A2} hold, and $\beta_j\in\mathcal P_\mu\cap(1,N)$ since every open interval containing $\beta_j$ has $\mu$-measure at least $c_j>0$.
\end{proof}

If we take $\eta=0$ and consider $\lambda>0$ in
\eqref{superposition problem}, we obtain a Brezis--Nirenberg-type equation,
namely
\begin{eqnarray}\label{BN superposition equation}
     -\Delta u+
     \int_{(0,N)}
     \left(I_{\beta}\star |u|^{q(\beta)}\right)
     |u|^{q(\beta)-2}u\,d\mu(\beta)
     =
     \lambda |u|^{q_\gamma-2}u
     +|u|^{2^*-2}u,
     \quad\text{in }\R^N.
\end{eqnarray}
For this equation, we prove that if $\lambda$ lies in a sufficiently small
left neighborhood of an eigenvalue of problem \eqref{e1}, then
\eqref{BN superposition equation} has multiple pairs of nontrivial solutions.
More precisely, if the corresponding eigenvalue has multiplicity $m$, then we
obtain $m$ distinct pairs of solutions. Here $(\lambda_k)_{k\geq1}$ denotes the sequence of eigenvalues of \eqref{e1} constructed in Section \ref{sec2} in the radial space $W=\X$ (see \eqref{def-lambdak}); we call them the {\sl radial eigenvalues} of \eqref{e1}. This is stated as follows.

\begin{theorem}\label{th2}
    Assume \ref{A1}--\ref{A3}, and suppose that
    \[
    \lambda_k=\cdots=\lambda_{k+m-1}<\lambda_{k+m}
    \]
    for some $k,m\geq1$. Then there exists $\delta_k>0$ such that equation
    \eqref{BN superposition equation} has $m$ distinct pairs of nontrivial radial solutions
    at positive energy levels for all
    $\lambda\in(\lambda_k-\delta_k,\lambda_k)$.
\end{theorem}

\begin{corollary}\label{cor2}
    Assume that $N\geq 3$, $0<a<b<N$ and $\gamma>\frac{(N-2)(b+2)}{4}$. Let
    $\{c_i\}_{i\in\N}$ be a sequence of nonnegative numbers
    such that
    \[
    \sum_{i=1}^\infty c_i<\infty,
    \]
    and let $\{\beta_i\}_{i\in\N}\subset [a,b]$. Suppose that for some $j\geq1$, one has $c_j>0$ and $\beta_j\in(1,N)$. Let $(\lambda_k)$ be the radial eigenvalues of \eqref{e1} with $\mu=\sum_i c_i\delta_{\beta_i}$, and suppose that $\lambda_k=\cdots=\lambda_{k+m-1}<\lambda_{k+m}$ for some $k,m\geq1$. Then there
    exists $\delta_k>0$ such that the equation
    \[
    -\Delta u+
    \sum_{i=1}^\infty c_i
    \left(I_{\beta_i}\star |u|^{q(\beta_i)}\right)
    |u|^{q(\beta_i)-2}u
    =
    \lambda |u|^{q_\gamma-2}u
    +|u|^{2^*-2}u,
    \quad\text{in }\,\R^N,
    \]
    has $m$ distinct pairs of nontrivial radial solutions at positive energy levels for all
    $\lambda\in(\lambda_k-\delta_k,\lambda_k)$.
\end{corollary}

\begin{remark}[On the assumptions]\label{rem-assumptions}
    Assumption \ref{A3} provides the compactness of the embedding $\X\hookrightarrow L^{q_\gamma}(\R^N)$ (Remark \ref{rem-qgamma}), which is needed for the eigenvalue theory of Section \ref{sec2} and for the local Palais--Smale condition. Assumption \ref{A2} is used to prove that Palais--Smale sequences are bounded, and to show that the energy of Palais--Smale sequences cannot be lost below the level $\frac1N S^{N/2}$; see Lemma \ref{ps-sup}.

    No further restriction on the parameters is needed. In particular:
    \begin{enumerate}[label=(\alph*)]
        \item the space $\x$ is uniformly convex, and hence reflexive, for every $\gamma>0$, i.e.\ whenever $q(\beta)>1$ for all $\beta$ (Section \ref{uniform-convexity}); no condition of the type $q(\beta)\geq2$ is required;
        \item weak solutions of \eqref{superposition problem} and \eqref{e1} are $C^{1,\vartheta}_{\rm loc}$ and satisfy a Pohozaev identity without any additional condition on $\gamma$, $a$ and $b$ (Lemmas \ref{regularity} and \ref{pohozaev}). This relies on a local estimate for the Riesz potentials $I_\beta\star|u|^{q(\beta)}$ which is uniform in $\beta\in[a,b]$ (Corollary \ref{far-field}).
    \end{enumerate}
    The assumption $r>q_\gamma$ in Theorem \ref{th1} cannot simply be dropped from the argument: when $r=q_\gamma$ the term $\eta|u|^{q_\gamma-2}u$ merely shifts $\lambda$ to $\lambda+\eta$, and the geometry used in the proof breaks down for large $\eta$.
\end{remark}

\begin{example}[A non-atomic superposition]\label{example}
    Let $N\geq3$, $0<a<b<N$ with $b>1$, $d\mu=\mathbf 1_{[a,b]}(\beta)\,d\beta$, and $\gamma>\frac{(N-2)(b+2)}{4}$. Then Theorems \ref{th1} and \ref{th2} apply to
    \[
    -\Delta u+\int_a^b\left(I_\beta\star|u|^{1+\frac{\beta+2}{2\gamma}}\right)|u|^{\frac{\beta+2}{2\gamma}-1}u\,d\beta=\lambda|u|^{q_\gamma-2}u+\eta|u|^{r-2}u+|u|^{2^*-2}u\quad\text{in }\R^N,
    \]
    which involves a continuum of Hartree terms with different orders $\beta$ and different powers $q(\beta)$. This problem is not a finite (or countable) sum of problems that have been previously studied. The corresponding energy space $Q_\gamma(\R^N)$ is a Luxemburg-type space whose modular is not homogeneous (it behaves like $\|u\|^{2q(a)}$ for small $u$ and like $\|u\|^{2q(b)}$ for large $u$), and it is not the intersection of finitely many Coulomb spaces. For this reason all estimates have to be made uniform with respect to $\beta\in[a,b]$, and the single-kernel techniques of \cite{MR3568051,MaMePe1} have to be combined with measure-theoretic arguments in the variable $\beta$ (see, e.g., Lemma \ref{uniform-constants} and Propositions \ref{weak convergence riesz} and \ref{weak riesz}).
\end{example}


\section{Scaling and scaled operators}\label{sec2}

This section will present the tools used to prove Theorems \ref{th1} and \ref{th2}. It is a portion of a novel critical point theory for scaled operators on a Banach space. A detailed presentation of this technique can be found in \cite{MePe1}.

\begin{definition}
    Let $W$ be a reflexive Banach space. A scaling on $W$ is a continuous mapping $W\times[0,\infty)\to W$, $(u,t)\mapsto u_t$ satisfying
    \begin{enumerate}[label=$(H_{\arabic*})$]
    \item $(u_{t_1})_{t_2}=u_{t_1t_2}$ for all $u\in W$ and $t_1,t_2\geq0,$
    \item $(\tau u)_t=\tau u_t$ for all $u\in W$, $\tau\in\R$ and $t\geq0$,
    \item $u_0=0$ and $u_1=u$ for all $u\in W$,
    \item $u_t$ is bounded on bounded sets of $W\times[0,\infty)$,
    \item there exists $\kappa>0$ such that $\|u_t\|=O(t^{\kappa})$ as $t\to\infty$, uniformly on bounded sets of $W$.
\end{enumerate}
\end{definition}
 Denote by $W^*$ the dual of $W$. Recall that $g\in C(W,W^*)$ is a potential operator if there is a functional $G\in C^1(W,\R)$, called a potential for $g$, such that $G^\prime=g$. By replacing $G$ with $G-G(0)$ if necessary, we may assume that $G(0)=0$.

 \begin{definition}
     Let $s>0$. A scaled operator (of degree $s$) is an odd potential operator $g\in C(W,W^*)$ that maps bounded sets into bounded sets and satisfies
\[g(u_t)v_t=t^sg(u)v,\quad\forall u,v\in W,\ t\geq0.
     \]
     We denote by $\A_s$ the class of odd scaled potential operators of degree $s$.
     \end{definition}
Let $A_s\in\A_s$ and let us denote by $I_s$ the potential of $A_s$ with $I_s(0)=0$. Of course $I_s$ is even, bounded on bounded sets, and satisfies the scaling property
\[I_s(u_t)=t^sI_s(u),\quad\forall u\in W, t\geq0\]
(see \cite[Proposition 2.2]{MePe1}).

We will consider the question of existence and multiplicity of solutions to equations of the form
\[
A_s(u)=f(u)\quad\text{in } W^*,
\]
where $f\in C(W,W^*)$ is a potential operator, and $A_s$ is a scaled operator satisfying
\begin{enumerate}[label=$(H_{\arabic*})$, start=6]
\item $A_s(u)u>0$ for all $u\in W\backslash\set{0}$,
\item every sequence $(u_j)$ in $W$ such that $u_j\weak u$ and $A_s(u_j)(u_j-u)\to0$ has a subsequence that converges strongly to $u$.
\end{enumerate}
 Solutions of the above equation coincide with critical points of the $C^1$-functional
\[
\Phi(u)=I_s(u)-F(u),\quad u\in W,
\]
where $F$ is the potential of $f$ with $F(0)=0$.

\subsection{Scaled eigenvalue problems}
Now let us consider the eigenvalue problem
 \begin{equation}\label{3-1}
     A_s(u)=\lambda B_s(u)\quad\text{in } W^*,
 \end{equation}
where $A_s,B_s\in\A_s$, $A_s$ satisfies $(H_6)$ and $(H_7)$, and $B_s$ satisfies
\begin{enumerate}[label=$(H_{\arabic*})$, start=8]
    \item $B_s(u)u>0$ for all $u\in W\backslash\set{0}$,
    \item if $u_j\weak u$ in $W$, then $B_s(u_j)\to B_s(u)$ in $W^*$,
\end{enumerate}
and $\lambda\in\R$. We say that $\lambda$ is an eigenvalue if there is a $u\in W\backslash\set{0}$, called an eigenfunction associated with $\lambda$, satisfying equation $\eqref{3-1}$. If that is the case, then $u_t$ is also an eigenfunction associated with $\lambda$ for any $t>0$ since
\[
A_s(u_t)v=A_s(u_t)(v_{1/t})_t=t^sA_s(u)v_{1/t}=t^s\lambda B_s(u)v_{1/t}=\lambda B_s(u_t)(v_{1/t})_t=\lambda B_s(u_t)v
\]
for all $v\in W$. We denote by $\sigma(A_s,B_s)$ the spectrum, i.e., the set of all eigenvalues, of the pair of scaled operators $(A_s,B_s)$. We have $\sigma(A_s,B_s)\subset (0,\infty)$ by $(H_6)$ and $(H_8)$.

Let us denote by $J_s$ the potential of $B_s$ with $J_s(0)=0$. We also assume that $I_s$ and $J_s$ satisfy
\begin{enumerate}[label=$(H_{\arabic*})$, start=10]
    \item $I_s$ is coercive, i.e., $I_s(u)\to\infty$ as $\|u\|\to\infty$,
    \item the equation $I_s(tu)=1$ has a unique positive solution $t$ for each $u\in W\backslash\set{0}$,
    \item every solution of $\eqref{3-1}$ satisfies
    \[
    I_s(u)=\lambda J_s(u).
    \]
\end{enumerate}

The eigenvalue problem $\eqref{3-1}$ has the following variational formulation. Let
\[
\Psi(u)=\frac{1}{J_s(u)},\quad u\in W\backslash\set{0},
\]
let
\[
\M_s=\set{u\in W:I_s(u)=1},
\]
and let $\w{\Psi}=\Psi|_{\M_s}$. Then $\M_s$ is a complete, symmetric, and bounded $C^1$-Finsler manifold, and eigenvalues of problem $\eqref{3-1}$ coincide with critical values of $\w{\Psi}$ (see \cite[Proposition 2.5]{MePe1}).

\subsubsection{Minimax eigenvalues}
\begin{definition}[Fadell and Rabinowitz \cite{MR0478189}] \label{Definition 301}
Let $\A$ denote the class of symmetric subsets of $W \setminus \set{0}$. For $A \in \A$, let $\overline{A} = A/\Z_2$ be the quotient space of $A$ with each $u$ and $-u$ identified, let $f : \overline{A} \to \RP^\infty$ be the classifying map of $\overline{A}$, and let $f^\ast : H^\ast(\RP^\infty) \to H^\ast(\overline{A})$ be the induced homomorphism of the Alexander-Spanier cohomology rings. The cohomological index of $A$ is defined by
\[
i(A) = \begin{cases}
0 & \text{if } A = \emptyset\\[5pt]
\sup \set{m \ge 1 : f^\ast(\omega^{m-1}) \ne 0} & \text{if } A \ne \emptyset,
\end{cases}
\]
where $\omega \in H^1(\RP^\infty)$ is the generator of the polynomial ring $H^\ast(\RP^\infty) = \Z_2[\omega]$.
\end{definition}

\begin{example}
The classifying map of the unit sphere $S^N$ in $\R^{N+1},\, N \ge 0$ is the inclusion $\RP^N \incl \RP^\infty$, which induces isomorphisms on the cohomology groups $H^l$ for $l \le N$, so $i(S^N) = N + 1$.
\end{example}

The following proposition summarizes the basic properties of the cohomological index.

\begin{proposition}[Fadell and Rabinowitz \cite{MR0478189}] \label{Proposition 300}
The index $i : \A \to \N \cup \set{0,\infty}$ has the following properties:
\begin{enumerate}
\item[$(i_1)$]Definiteness: $i(A) = 0$ if and only if $A = \emptyset$.
\item[$(i_2)$] Monotonicity: If there is an odd continuous map from $A$ to $B$ (in particular, if $A \subset B$), then $i(A) \le i(B)$. Thus, equality holds when the map is an odd homeomorphism.
\item[$(i_3)$] Dimension: $i(A) \le \dim W$.
\item[$(i_4)$] Continuity: If $A$ is closed, then there is a closed neighborhood $N \in \A$ of $A$ such that $i(N) = i(A)$. When $A$ is compact, $N$ may be chosen to be a $\delta$-neighborhood $N_\delta(A) = \set{u \in W : \dist{u}{A} \le \delta}$.
\item[$(i_5)$] Subadditivity: If $A$ and $B$ are closed, then $i(A \cup B) \le i(A) + i(B)$.
\item[$(i_6)$] Stability: If $\Sigma A$ is the suspension of $A \ne \emptyset$, obtained as the quotient space of $A \times [-1,1]$ with $A \times \set{1}$ and $A \times \set{-1}$ collapsed to different points, then $i(\Sigma A) = i(A) + 1$.
\item[$(i_7)$] Piercing property: If $C$, $C_0$, and $C_1$ are closed and $\varphi : C \times [0,1] \to C_0 \cup C_1$ is a continuous map such that $\varphi(-u,t) = - \varphi(u,t)$ for all $(u,t) \in C \times [0,1]$, $\varphi(C \times [0,1])$ is closed, $\varphi(C \times \set{0}) \subset C_0$, and $\varphi(C \times \set{1}) \subset C_1$, then $i(\varphi(C \times [0,1]) \cap C_0 \cap C_1) \ge i(C)$.
\item[$(i_8)$] Neighborhood of zero: If $U$ is a bounded closed symmetric neighborhood of $0$, then $i(\bdry{U}) = \dim W$.
\end{enumerate}
\end{proposition}

Let $\F$ denote the class of symmetric subsets of $\M_s$. For $k\geq1$, let
\[
\F_k=\set{M\in\F:i(M)\geq k}
\]
and set
\begin{equation}\label{def-lambdak}
\lambda_k:=\inf_{M\in\F_k}\sup_{u\in M}\w{\Psi}(u).
\end{equation}
We have the following theorem (see Mercuri and Perera \cite[Theorem 2.10]{MePe1}; cf.\ also \cite[Propositions 3.52 and 3.53]{MR2640827} for the analogous results for $p$-Laplacian type operators).
\begin{theorem}\label{Theorem 301}
    Assume $(H_1)$--$(H_{12})$. Then $\lambda_k\nearrow \infty$ is a sequence of eigenvalues of $\eqref{3-1}$.
    \begin{enumroman}
        \item\label{301-1} The first eigenvalue is given by
        \[
        \lambda_1=\min_{u\in\M_s}\w{\Psi}(u)>0.
        \]
        \item\label{301-2} If $\lambda_k=\dotsb=\lambda_{k+m-1}=\lambda$, and $E_\lambda$ is the set of eigenfunctions associated with $\lambda$ that lie on $\M_s$, then
        \[i(E_\lambda)\geq m.\]
        \item\label{301-3} If $\lambda_k<\lambda<\lambda_{k+1}$, then
        \[i(\w{\Psi}^{\lambda_k})=i(\M_s\backslash\w{\Psi}_\lambda)=i(\w{\Psi}^\lambda)=i(\M_s\backslash\w{\Psi}_{\lambda_{k+1}})=k,\]
        where $\w{\Psi}^c=\set{u\in\M_s:\w{\Psi}(u)\leq c}$ and $\w{\Psi}_c=\set{u\in\M_s:\w{\Psi}(u)\geq c}$ for $c\in\R$.
    \end{enumroman}
\end{theorem}

We will also use the following consequence of Theorem \ref{Theorem 301}.

\begin{lemma}\label{index-below}
    Assume $(H_1)$--$(H_{12})$. Then $i(\M_s\setminus\w\Psi_{\lambda_k})\leq k-1$ for every $k\geq1$.
\end{lemma}
\begin{proof}
    Let $l\leq k$ be the smallest index with $\lambda_l=\lambda_k$. If $l=1$, then $\M_s\setminus\w\Psi_{\lambda_k}$ is the set of $u\in\M_s$ with $\w\Psi(u)<\lambda_1$, which is empty by Theorem \ref{Theorem 301} \ref{301-1}, so its index is $0\leq k-1$ by $(i_1)$. If $l\geq2$, then $\lambda_{l-1}<\lambda_l$, and Theorem \ref{Theorem 301} \ref{301-3} (applied with $k$ replaced by $l-1$ and any $\lambda\in(\lambda_{l-1},\lambda_l)$) gives $i(\M_s\setminus\w\Psi_{\lambda_k})=i(\M_s\setminus\w\Psi_{\lambda_l})=l-1\leq k-1$.
\end{proof}

\subsection{Multiplicity based on scaling}

Let $\Phi\in C^1(W,\R)$ be an even functional, i.e., $\Phi(-u)=\Phi(u)$ for all $u\in W$. Assume that there exists $c^*>0$ such that $\Phi$ satisfies the $(PS)_c$ condition for all $c\in(0,c^*)$. Let $\Gamma$ denote the group of odd homeomorphisms of $W$ that are the identity outside $\Phi^{-1}(0,c^*)$. Let $\A^*$ denote the class of symmetric subsets of $W$, and let
\begin{eqnarray*}
    \M_{s,\rho}=\set{u\in W: I_s(u)=\rho^s}=\set{u_\rho:u\in\M_s}
\end{eqnarray*}
for $\rho>0.$
\begin{definition}[Benci \cite{MR84c:58014}]
    The pseudo-index of $M\in\A^*$ related to $i$, $\M_{s,\rho}$, and $\Gamma$ is defined by
    \begin{eqnarray*}
        i^*(M)=\min_{h\in\Gamma} i(h(M)\cap\M_{s,\rho}).
    \end{eqnarray*}
\end{definition}
We have the following multiplicity result.

\begin{theorem}[Mercuri \& Perera \cite{MePe1}]\label{m1}
    Let $A_0$ and $B_0$ be symmetric subsets of $\M_s$ such that $A_0$ is compact, $B_0$ is closed, and
    \begin{eqnarray}\label{m2}
        i(A_0)\geq k+m-1,\quad i(\M_s\backslash B_0)\leq k-1
    \end{eqnarray}
    for some $k,m\geq1$. Let $R>\rho>0$ and let
    \begin{eqnarray*}
        X&=&\set{u_t:u\in A_0, 0\leq t\leq R},\\
        A&=&\set{u_R:u\in A_0},\\
        B&=&\set{u_{\rho}:u\in B_0}.
        \end{eqnarray*}
        Assume that
        \begin{eqnarray}\label{m1-2}
            \sup_{u\in A}\Phi(u)\leq0<\inf_{u\in B}\Phi(u),\quad \sup_{u\in X}\Phi(u)<c^*.
        \end{eqnarray}
        For $j=k,\dots,k+m-1$, let
        \begin{eqnarray*}
            \A_j^*=\set{M\in\A^*:M\text{ is compact and }i^*(M)\geq j}
        \end{eqnarray*}
        and set
        \begin{eqnarray*}
            c^*_j=\inf_{M\in\A_j^*}\max_{u\in M}\Phi(u).
        \end{eqnarray*}
        Then $0<c_k^*\leq\cdots\leq c_{k+m-1}^*<c^*$, each $c_j^*$ is a critical value of $\Phi$, and $\Phi$ has $m$ distinct pairs of associated critical points.
\end{theorem}

\begin{corollary}\label{m2_corollary}
    Let $A_0$ be a compact symmetric subset of $\M_s$ with $i(A_0)=m\geq1$, let $R>\rho>0$, and let
    \begin{eqnarray*}
        A=\set{u_R:u\in A_0},\quad X=\set{u_t:u\in A_0,0\leq t\leq R}.
    \end{eqnarray*}
    Assume that
    \begin{eqnarray*}
        \sup_{u\in A}\Phi(u)\leq0<\inf_{u\in\M_{s,\rho}}\Phi(u),\quad\sup_{u\in X}\Phi(u)<c^*.
    \end{eqnarray*}
    Then $\Phi$ has $m$ distinct pairs of critical points at levels in $(0,c^*)$.
\end{corollary}


\section{Superposed Coulomb--Sobolev spaces}\label{sec3}
This section is devoted to the construction and study of the energy space associated with the problem \eqref{superposition problem}. Throughout this section we fix $0<a<b<N$, a finite positive Borel measure $\mu\not\equiv0$ on $(0,N)$ such that $\mu(E)=0$ for every Borel set $E\subset (0,N)\backslash [a,b]$, and $\gamma>0$. For each $\beta\in(0,N)$ we define
    \begin{eqnarray}\label{beta}
        q(\beta)=1+\dfrac{\beta+2}{2\gamma},
    \end{eqnarray}
so that $1<q(a)\leq q(\beta)\leq q(b)$ for all $\beta\in[a,b]$. Integrals with respect to $d\mu$ over $(0,N)$ are therefore integrals over $[a,b]$. Unless otherwise stated, all constants below are allowed to depend on $N$, $a$, $b$, $\gamma$ and $\mu$.

\subsection{Riesz energies and the superposed modular}

For $\beta\in(0,N)$ and measurable functions $f,g:\R^N\to[0,\infty]$ we set
\[
\Bn_\beta(f,g):=\iint_{\R^N\times\R^N}I_\beta(x-y)f(x)g(y)\,dx\,dy\in[0,\infty],\qquad \En_\beta(f):=\Bn_\beta(f,f).
\]
By the semigroup property of the Riesz potentials and Tonelli's theorem,
\begin{equation}\label{s2}
\Bn_\beta(f,g)=\int_{\R^N}\seq{I_{\beta/2}\star f}\seq{I_{\beta/2}\star g}dx,\qquad \En_\beta(f)=\int_{\R^N}\left|I_{\beta/2}\star f\right|^2dx.
\end{equation}

\begin{lemma}\label{energy-props}
    Let $\beta\in(0,N)$, let $f,g\geq0$ be measurable, and let $c\geq0$ and $\theta\in[0,1]$. Then:
    \begin{enumroman}
        \item\label{ep1} $\En_\beta(cf)=c^2\En_\beta(f)$, and $f\leq g$ implies $\En_\beta(f)\leq\En_\beta(g)$;
        \item\label{ep2} $\Bn_\beta(f,g)\leq\En_\beta(f)^{1/2}\En_\beta(g)^{1/2}$, and
        \[
        \En_\beta(f)+\En_\beta(g)\leq\En_\beta(f+g)\leq2\En_\beta(f)+2\En_\beta(g),\qquad \En_\beta(\theta f+(1-\theta)g)\leq\theta\En_\beta(f)+(1-\theta)\En_\beta(g);
        \]
        \item\label{ep3} $\En_\beta(f^\theta g^{1-\theta})\leq\En_\beta(f)^\theta\En_\beta(g)^{1-\theta}$.
    \end{enumroman}
\end{lemma}
\begin{proof}
    \ref{ep1} is immediate from the positivity of $I_\beta$. In \ref{ep2}, the first inequality is the Cauchy--Schwarz inequality in $L^2(\R^N)$ applied to \eqref{s2}; the second one follows from $\En_\beta(f+g)=\En_\beta(f)+2\Bn_\beta(f,g)+\En_\beta(g)$ and $0\leq2\Bn_\beta(f,g)\leq\En_\beta(f)+\En_\beta(g)$; the third one follows from the convexity of $F\mapsto\|F\|_{L^2}^2$ and the linearity of $f\mapsto I_{\beta/2}\star f$. For \ref{ep3} we apply H\"older's inequality with exponents $1/\theta$ and $1/(1-\theta)$ on $\R^N\times\R^N$ with respect to the measure $I_\beta(x-y)\,dx\,dy$ to the functions $F(x,y)=f(x)f(y)$ and $G(x,y)=g(x)g(y)$, noting that $(f^\theta g^{1-\theta})(x)(f^\theta g^{1-\theta})(y)=F^\theta G^{1-\theta}$.
\end{proof}

For a measurable function $u:\R^N\to\R$ and $\beta\in(0,N)$ we define
\[
\rho_\beta(u):=\En_\beta\seq{|u|^{q(\beta)}}=\int_{\R^N}\left|I_{\beta/2}\star|u|^{q(\beta)}\right|^2dx=\iint_{\R^N\times\R^N}\dfrac{A_\beta|u(x)|^{q(\beta)}|u(y)|^{q(\beta)}}{|x-y|^{N-\beta}}dxdy
\]
and
    \begin{eqnarray*}
        \rho(u):=\int_{(0,N)}\rho_\beta(u)\,d\mu(\beta)=\int_{(0,N)}\int_{\R^N}\left|I_{\beta/2}\star|u|^{q(\beta)}\right|^2dx\,d\mu(\beta)\in[0,\infty].
    \end{eqnarray*}
Since the function $(x,y,\beta)\mapsto A_\beta|x-y|^{\beta-N}|u(x)|^{q(\beta)}|u(y)|^{q(\beta)}$ is nonnegative and Borel measurable, the map $\beta\mapsto\rho_\beta(u)$ is Borel measurable by Tonelli's theorem, so that $\rho(u)$ is well defined. Note that $\rho_\beta(tu)=|t|^{2q(\beta)}\rho_\beta(u)$ for $t\in\R$. Since $s\mapsto|s|^{q(\beta)}$ is convex, Lemma \ref{energy-props} \ref{ep1}--\ref{ep2} show that $\rho_\beta$ and $\rho$ are convex.

\begin{definition}[Superposed Coulomb spaces]\label{def.coul.space}
    We define the {\sl superposed Coulomb space} $Q_\gamma(\R^N)$, with respect to $\gamma$ and the measure $\mu$, as the set of all measurable functions $u:\R^N\to\R$ such that $\rho(u/\tau)<+\infty$
    for some $\tau=\tau(u)>0$, and we set
\begin{eqnarray}\label{normQ}
    \|u\|_{Q_{\gamma}}:=\inf\set{\tau>0:\rho(u/\tau)\leq1}.
\end{eqnarray}
\end{definition}

\begin{lemma}\label{modular}
    Let $u\in Q_\gamma(\R^N)$. Then:
    \begin{enumroman}
        \item\label{mod1} $\rho(tu)<\infty$ for every $t\geq0$, and
        \[
        \min\set{t^{2q(a)},t^{2q(b)}}\rho(u)\leq\rho(tu)\leq\max\set{t^{2q(a)},t^{2q(b)}}\rho(u),\qquad t\geq0;
        \]
        \item\label{mod2} $\rho(u)=0$ if and only if $u=0$ a.e. in $\R^N$;
        \item\label{mod3} the map $t\mapsto\rho(tu)$ is continuous and nondecreasing on $[0,\infty)$; if $u\neq0$, it is strictly increasing and $\rho(tu)\to\infty$ as $t\to\infty$.
    \end{enumroman}
\end{lemma}
\begin{proof}
    Since $\rho(tu)=\int t^{2q(\beta)}\rho_\beta(u)\,d\mu$ and $q(a)\leq q(\beta)\leq q(b)$ on $[a,b]$, the two-sided inequality in \ref{mod1} holds (in $[0,\infty]$). If $\tau>0$ is such that $\rho(u/\tau)<\infty$, then $\rho(tu)=\rho(t\tau\cdot u/\tau)\leq\max\set{(t\tau)^{2q(a)},(t\tau)^{2q(b)}}\rho(u/\tau)<\infty$. This proves \ref{mod1}.

    If $\rho(u)=0$, then $\rho_\beta(u)=0$ for $\mu$-a.e.\ $\beta$; since $\mu\not\equiv0$, there exists $\beta\in[a,b]$ with $\iint A_\beta|x-y|^{\beta-N}|u(x)|^{q(\beta)}|u(y)|^{q(\beta)}dxdy=0$, which forces $|u(x)||u(y)|=0$ for a.e.\ $(x,y)$, i.e., $u=0$ a.e. This proves \ref{mod2}.

    Continuity in \ref{mod3} follows from the dominated convergence theorem, since $t^{2q(\beta)}\rho_\beta(u)\leq\max\set{1,T^{2q(b)}}\rho_\beta(u)$ for $t\in[0,T]$. If $0\leq t_1<t_2$, then $\rho(t_2u)-\rho(t_1u)=\int(t_2^{2q(\beta)}-t_1^{2q(\beta)})\rho_\beta(u)\,d\mu\geq0$, with strict inequality if $u\neq0$, because then $\rho_\beta(u)>0$ on a set of positive $\mu$-measure by \ref{mod2}. Finally $\rho(tu)\geq t^{2q(a)}\rho(u)\to\infty$ as $t\to\infty$ if $u\neq0$.
\end{proof}

\begin{proposition}\label{norm}
    $Q_\gamma(\R^N)$ is a vector space and $\|\cdot\|_{Q_{\gamma}}$ defines a norm on it.
\end{proposition}
\begin{proof}
    If $u,v\in\Q$ and $c\in\R$, then $cu\in\Q$ by Lemma \ref{modular} \ref{mod1}, and $\rho((u+v)/2)\leq(\rho(u)+\rho(v))/2<\infty$ by convexity, so that $u+v\in\Q$. By Lemma \ref{modular}, $\rho(u/\tau)\to0$ as $\tau\to\infty$, so $\normq{u}<\infty$.
    If $\normq{u}=0$, we can take a sequence $\tau_n\searrow0$ with $\rho(u/\tau_n)\leq1$, and Lemma \ref{modular} \ref{mod1} gives $\rho(u)\leq\tau_n^{2q(a)}\rho(u/\tau_n)\leq\tau_n^{2q(a)}$ for $\tau_n<1$; hence $\rho(u)=0$ and $u=0$ a.e.\ by Lemma \ref{modular} \ref{mod2}.
    For $c\neq0$,
    \begin{eqnarray*}
        \normq{cu}=\inf\set{\tau>0:\rho(|c|u/\tau)\leq1}=|c|\inf\set{\tau/|c|>0:\rho(u/(\tau/|c|))\leq1}=|c|\normq{u},
    \end{eqnarray*}
    since $\rho$ is even. Finally, let $u,v\in\Q$, $\tau_u>\normq{u}$ and $\tau_v>\normq{v}$, so that $\rho(u/\tau_u)\leq1$ and $\rho(v/\tau_v)\leq1$. With $\tau=\tau_u+\tau_v$, the convexity of $\rho$ gives
    \[
 \rho\left(\frac{u+v}{\tau}\right)=\rho\seq{\frac{\tau_u}{\tau}\frac{u}{\tau_u}+\frac{\tau_v}{\tau}\frac{v}{\tau_v}}\leq\frac{\tau_u}{\tau}\rho\seq{\frac{u}{\tau_u}}+\frac{\tau_v}{\tau}\rho\seq{\frac{v}{\tau_v}}\leq1,
    \]
    so that $\normq{u+v}\leq\tau_u+\tau_v$. Taking the infimum over $\tau_u$ and $\tau_v$ concludes the proof.
\end{proof}

The following lemma is also known as the {\sl unit ball property}, see for example \cite[Chapter 2]{diening2011lebesgue}.

\begin{lemma}\label{ball property} Let $u\in\Q$. Then:
   \begin{enumroman}
       \item\label{ball propety1} $\normq{u}\leq1$ and $\rho(u)\leq1$ are equivalent;
       \item\label{ball propety2}  $\normq{u}<1$ and $\rho(u)<1$ are equivalent;
       \item\label{ball propety3}  $\normq{u}=1$ and $\rho(u)=1$ are equivalent.
   \end{enumroman}
   Moreover, $\rho(u/\normq{u})=1$ for every $u\neq0$.
\end{lemma}
\begin{proof}
    The case $u=0$ is trivial, so let $u\neq0$ and $\varphi(\tau):=\rho(u/\tau)$ for $\tau>0$. By Lemma \ref{modular}, $\varphi$ is continuous and strictly decreasing, with $\varphi(\tau)\to\infty$ as $\tau\to0^+$ and $\varphi(\tau)\to0$ as $\tau\to\infty$. Hence there is a unique $\tau_0>0$ with $\varphi(\tau_0)=1$, and $\set{\tau>0:\varphi(\tau)\leq1}=[\tau_0,\infty)$, so that $\normq{u}=\tau_0$ and $\rho(u/\normq{u})=1$. Since $\rho(u)=\varphi(1)$ and $\varphi$ is strictly decreasing, $\rho(u)\leq1$, $\rho(u)<1$, $\rho(u)=1$ are equivalent to $\tau_0\leq1$, $\tau_0<1$, $\tau_0=1$, respectively.
\end{proof}

Combining Lemmas \ref{modular} and \ref{ball property} (apply Lemma \ref{modular} \ref{mod1} to $u/\normq{u}$ with $t=\normq{u}$) we obtain the following two-sided relation between the norm and the modular, which will be used throughout the paper.

\begin{corollary}\label{norm-modular}
For every $u\in\Q$,
\begin{equation}\label{two-sided}
    \min\set{\normq{u}^{2q(a)},\normq{u}^{2q(b)}}\leq\rho(u)\leq\max\set{\normq{u}^{2q(a)},\normq{u}^{2q(b)}}.
\end{equation}
Consequently:
\begin{enumroman}
    \item\label{nm1} $\normq{u_n-u}\to0$ if and only if $\rho(u_n-u)\to0$;
    \item\label{nm2} a subset of $\Q$ is bounded if and only if $\rho$ is bounded on it;
    \item\label{nm3} if $\varepsilon\in(0,1)$ and $\rho(u)\leq1-\varepsilon$, then $\normq{u}\leq(1-\varepsilon)^{1/(2q(b))}$;
    \item\label{nm4} $\rho(2u)\leq2^{2q(b)}\rho(u)$.
\end{enumroman}
\end{corollary}


\subsection{Local estimates and uniform constants}

We will use repeatedly the following local estimate, see \cite[Proposition 2.3]{MR3568051}. We include the short proof in order to make the dependence of the constant on $\beta$ explicit.

\begin{proposition}\label{den}
    Let $\beta\in(0,N)$ and $C_\beta:=\dfrac{2^{N-\beta/2}}{A_{\beta/2}|B_1|^{1/2}}$. For every measurable $f\geq0$, $x_0\in\R^N$ and $R>0$,
    \[
    \int_{B_R(x_0)}f\,dx\leq C_\beta R^{\frac{N-\beta}{2}}\seq{\int_{B_R(x_0)}\left|I_{\beta/2}\star f\right|^2dx}^{1/2}.
    \]
    In particular $\int_{B_R(x_0)}|u|^{q(\beta)}dx\leq C_\beta R^{\frac{N-\beta}{2}}\rho_\beta(u)^{1/2}$ for every measurable $u$, and $\sup_{\beta\in[a,b]}C_\beta<\infty$.
\end{proposition}
\begin{proof}
    For $x,y\in B_R(x_0)$ we have $I_{\beta/2}(x-y)\geq A_{\beta/2}(2R)^{\beta/2-N}$, hence
    \[
    I_{\beta/2}\star f(x)\geq A_{\beta/2}(2R)^{\beta/2-N}\int_{B_R(x_0)}f\qquad\text{for all }x\in B_R(x_0).
    \] Squaring and integrating over $B_R(x_0)$ gives the claim. The last assertion follows from the continuity and positivity of $\beta\mapsto A_{\beta/2}$ on $(0,N)$.
\end{proof}

Summing the estimate of Proposition \ref{den} over dyadic annuli we obtain the following tail estimate.

\begin{lemma}\label{tail}
    Let $\beta\in(0,N)$, $R>0$ and $\kappa>\frac{N-\beta}{2}$. Then, for every measurable $u$,
    \[
    \int_{\R^N\setminus B_R}\frac{|u(y)|^{q(\beta)}}{|y|^{\kappa}}dy\leq \frac{2^{\frac{N-\beta}{2}}C_\beta}{1-2^{\frac{N-\beta}{2}-\kappa}}\,R^{\frac{N-\beta}{2}-\kappa}\rho_\beta(u)^{1/2}.
    \]
\end{lemma}
\begin{proof}
    Let $D_k=\set{2^kR\leq|y|<2^{k+1}R}$, $k\geq0$. By Proposition \ref{den},
    \[
    \int_{D_k}\frac{|u|^{q(\beta)}}{|y|^{\kappa}}dy\leq(2^kR)^{-\kappa}\int_{B_{2^{k+1}R}}|u|^{q(\beta)}dy\leq C_\beta(2^kR)^{-\kappa}(2^{k+1}R)^{\frac{N-\beta}{2}}\rho_\beta(u)^{1/2},
    \]
    and the conclusion follows by summing the geometric series.
\end{proof}

\begin{corollary}[Far-field bound]\label{far-field}
    For every $R>0$ there exists $C=C(N,a,b,R)>0$ such that
    \[
    \seq{I_\beta\star\seq{\mathbf 1_{\R^N\setminus B_{2R}}|u|^{q(\beta)}}}(x)\leq C\rho_\beta(u)^{1/2}\qquad\text{for all }x\in B_R,\ \beta\in[a,b],
    \]
    and every measurable $u$.
\end{corollary}
\begin{proof}
    If $|x|<R$ and $|y|\geq2R$, then $|x-y|\geq|y|/2$, so that $I_\beta(x-y)\leq2^{N-\beta}A_\beta|y|^{\beta-N}$. Now apply Lemma \ref{tail} with $2R$ in place of $R$ and $\kappa=N-\beta$, noting that $\kappa-\frac{N-\beta}{2}=\frac{N-\beta}{2}\geq\frac{N-b}{2}>0$ and that $A_\beta$, $C_\beta$ are bounded on $[a,b]$.
\end{proof}

For $\beta\in(0,N)$ and $p\geq1$ we denote by $E^{\beta,p}(\R^N)$ the Coulomb--Sobolev space of \cite{MR3568051}, that is, the space of weakly differentiable functions $u$ such that
\[
\|u\|_{E^{\beta,p}}:=\seq{\int_{\R^N}|\nabla u|^2dx+\seq{\int_{\R^N}\left|I_{\beta/2}\star|u|^p\right|^2dx}^{1/p}}^{1/2}<\infty,
\]
and by $E^{\beta,p}_{\text{rad}}(\R^N)$ its subspace of radial functions. We set
\[
G_\beta(u):=\|\nabla u\|_{L^2}+\rho_\beta(u)^{\frac{1}{2q(\beta)}},
\]
so that $\|u\|_{E^{\beta,q(\beta)}}\leq G_\beta(u)\leq\sqrt2\,\|u\|_{E^{\beta,q(\beta)}}$. Note that if $u\in\x$ (resp.\ $u\in\X$), then $\rho_\beta(u)<\infty$ for $\mu$-a.e.\ $\beta$, and for every such $\beta$ we have $u\in E^{\beta,q(\beta)}(\R^N)$ (resp.\ $u\in E^{\beta,q(\beta)}_{\text{rad}}(\R^N)$). Inequalities of \cite{MR3568051} hold for each fixed $\beta$ with constants depending on $\beta$; the following lemma allows us to integrate them with respect to $\mu$.

\begin{lemma}[Uniform constants]\label{uniform-constants}
    Let $\ell\in[1,\infty)$, let $J\subset[a,b]$ be a Borel set with $\mu(J)>0$, and let $\mathcal V$ denote either $C_c^\infty(\R^N)$ or the space $C^\infty_{c,\text{rad}}(\R^N)$ of radial functions in $C_c^\infty(\R^N)$. Assume that for every $\beta\in J$ there exists $C(\beta)<\infty$ such that
    \begin{equation}\label{uc-hyp}
        \|u\|_{L^\ell}\leq C(\beta)\,G_\beta(u)\qquad\text{for all }u\in\mathcal V.
    \end{equation}
    Then there exist $M>0$ and a Borel set $J'\subset J$ with $\mu(J')\geq\mu(J)/2$ such that $\|u\|_{L^\ell}\leq M\,G_\beta(u)$ for every $\beta\in J'$ and every $u\in E^{\beta,q(\beta)}(\R^N)$ (resp.\ $u\in E_{\text{rad}}^{\beta,q(\beta)}(\R^N)$). Consequently,
    \begin{equation}\label{uc-concl}
        \|u\|_{L^\ell}\leq M\seq{1+\frac{2}{\mu(J)}}\|u\|_{\x}\qquad\text{for all }u\in\x\ (\text{resp. }u\in\X).
    \end{equation}
\end{lemma}
\begin{proof}
    For $k\in\N$ the space $\set{v\in\mathcal V:\text{supp}\,v\subset B_k}$ is separable with respect to the $C^1$-norm, so there exists a countable set $\D\subset\mathcal V$ such that for every $v\in\mathcal V$ there are $d_j\in\D$, all supported in a fixed ball $B_k$, with $d_j\to v$ in $C^1(\R^N)$. For such a sequence we have $\|d_j\|_{L^\ell}\to\|v\|_{L^\ell}$, $\|\nabla d_j\|_{L^2}\to\|\nabla v\|_{L^2}$ and, for every $\beta\in[a,b]$, $\rho_\beta(d_j-v)\leq\|d_j-v\|_{L^\infty}^{2q(\beta)}\En_\beta(\mathbf 1_{B_k})\to0$. Since $\rho_\beta^{1/(2q(\beta))}$ is a norm \cite[Proposition 2.1]{MR3568051}, we get $G_\beta(d_j)\to G_\beta(v)$.

    Let $C^*(\beta)$ be the best constant in \eqref{uc-hyp}. By the above, $C^*(\beta)=\sup_{d\in\D\setminus\set{0}}\|d\|_{L^\ell}/G_\beta(d)$. For a fixed $d\in\D\setminus\{0\}$, the function $\beta\mapsto\rho_\beta(d)$ is continuous on $[a,b]$ by dominated convergence (the integrand $A_\beta|x-y|^{\beta-N}|d(x)|^{q(\beta)}|d(y)|^{q(\beta)}$ is continuous in $\beta$ and bounded by $C(1+|x-y|^{a-N})\mathbf 1_{B_k}(x)\mathbf 1_{B_k}(y)$), so that $\beta\mapsto\|d\|_{L^\ell}/G_\beta(d)$ is continuous. Hence $C^*$ is lower semicontinuous, the sets $J_M:=\set{\beta\in J:C^*(\beta)\leq M}$ are Borel, and $J_M\nearrow J$ as $M\to\infty$. We choose $M$ with $\mu(J_M)\geq\mu(J)/2$ and set $J'=J_M$.

    Now let $\beta\in J'$ and $u\in E^{\beta,q(\beta)}(\R^N)$ (resp.\ $E^{\beta,q(\beta)}_{\text{rad}}(\R^N)$). By \cite[Proposition 2.6]{MR3568051} (whose proof preserves radial symmetry when a radial mollifier is used, see also Proposition \ref{den1} below), there are $u_j\in\mathcal V$ with $\|u_j-u\|_{E^{\beta,q(\beta)}}\to0$; in particular $G_\beta(u_j)\to G_\beta(u)$ and, by Proposition \ref{den}, $u_j\to u$ in $L^1_{\text{loc}}(\R^N)$, hence a.e.\ along a subsequence. Fatou's lemma gives $\|u\|_{L^\ell}\leq\liminf\|u_j\|_{L^\ell}\leq M\,G_\beta(u)$.

    Finally, let $u\in\x$ (resp.\ $\X$) with $\normq{u}=1$, so that $\rho(u)=1$. For $\mu$-a.e.\ $\beta\in J'$ we have $\rho_\beta(u)<\infty$, and therefore $\|u\|_{L^\ell}\leq M(\|\nabla u\|_{L^2}+\rho_\beta(u)^{1/(2q(\beta))})\leq M(\|\nabla u\|_{L^2}+1+\rho_\beta(u))$. Integrating over $J'$ and dividing by $\mu(J')$ gives $\|u\|_{L^\ell}\leq M(\|\nabla u\|_{L^2}+1+1/\mu(J'))$. Applying this to $u/\normq{u}$ for general $u\neq0$ yields \eqref{uc-concl}.
\end{proof}


\subsection{Superposed Coulomb--Sobolev spaces and completeness}

\begin{definition}
    We define the {\sl superposed Coulomb--Sobolev} space $\x$ as the intersection of $\Q$ and $\mathcal{D}^{1,2}(\R^N)$, i.e.,
\begin{eqnarray*}
    \x:=\Q\cap\mathcal{D}^{1,2}(\R^N),
\end{eqnarray*}
with the norm
\begin{eqnarray}\label{normX}
    \norm{u}:=\|\nabla u\|_{L^2}+\normq{u},
\end{eqnarray}
and we denote by $\X$ the subspace of radial functions in $\x$. Here $\mathcal D^{1,2}(\R^N)=\set{u\in L^{2^*}(\R^N):\nabla u\in L^2(\R^N)}$ when $N\geq3$.
\end{definition}

\begin{theorem}\label{emb}
    There exists a constant $K=K(N,\gamma,\mu)>0$ such that
    \begin{eqnarray*}
        \|u\|_{L^{q_\gamma}(\R^N)}\leq K\norm{u}
    \end{eqnarray*}
    for all $u\in\x$. Consequently, if $N\geq3$ and $q_\gamma\leq2^*$, then $\x\hookrightarrow L^{\ell}(\R^N)$ continuously for every $\ell\in[q_\gamma,2^*]$.
\end{theorem}
\begin{proof}
    By \cite[Proposition 3.1]{MR3568051}, for every $\beta\in(0,N)$ there exists $C=C(N,\beta)>0$ such that
    \begin{eqnarray*}
        \om|u|^{2\frac{2q(\beta)+\beta}{2+\beta}}dx\leq C\seq{\om|\nabla u|^2dx}^{\frac{\beta}{2+\beta}}\seq{\om|I_{\beta/2}\star|u|^{q(\beta)}|^2dx}^{\frac{2}{2+\beta}}
    \end{eqnarray*}
    for all $u\in E^{\beta,q(\beta)}(\R^N)$. A direct computation shows that
    \[
    2\dfrac{2q(\beta)+\beta}{2+\beta}=2+\dfrac{2}{\gamma}=q_\gamma,
    \]
    so that $\|u\|_{L^{q_\gamma}}^{q_\gamma}\leq C\|\nabla u\|_{L^2}^{\frac{2\beta}{2+\beta}}\seq{\rho_\beta(u)^{\frac{1}{2q(\beta)}}}^{\frac{4q(\beta)}{2+\beta}}\leq C\,G_\beta(u)^{q_\gamma}$, because $\frac{2\beta}{2+\beta}+\frac{4q(\beta)}{2+\beta}=q_\gamma$. Hence \eqref{uc-hyp} holds with $\ell=q_\gamma$, $J=[a,b]$ and $\mathcal V=C_c^\infty(\R^N)$, and the first assertion follows from Lemma \ref{uniform-constants}. The second one follows from the Sobolev inequality and interpolation.
\end{proof}

\begin{theorem}\label{banach}
    $\x$ is a Banach space and $\X$ is a closed subspace of $\x$.
\end{theorem}
\begin{proof}
    Let $\s{u_n}\subset\x$ be a Cauchy sequence. Then $(\nabla u_n)$ is a Cauchy sequence in $L^2(\R^N,\R^N)$, hence it converges to some $G\in L^2(\R^N,\R^N)$, and by the Sobolev inequality and Theorem \ref{emb}, $(u_n)$ converges in $L^{2^*}(\R^N)$ and in $L^{q_\gamma}(\R^N)$ to some $u$. For $\varphi\in C_c^\infty(\R^N)$ we have $\int u\,\partial_i\varphi=\lim\int u_n\,\partial_i\varphi=-\lim\int\partial_iu_n\,\varphi=-\int G_i\varphi$, so that $\nabla u=G$ and $u\in\mathcal D^{1,2}(\R^N)$, with $\nabla u_n\to\nabla u$ in $L^2$. Passing to a subsequence, $u_n\to u$ a.e.\ in $\R^N$. By Fatou's lemma (applied first in $y$ inside the convolution, then in $x$ and $\beta$), for every $n$,
    \[
    \rho(u_n-u)\leq\liminf_{m\to\infty}\rho(u_n-u_m)\leq\sup_{m\geq n}\rho(u_n-u_m).
    \]
    Since $\normq{u_n-u_m}\to0$ as $n,m\to\infty$, Corollary \ref{norm-modular} \ref{nm1} shows that the right-hand side tends to $0$ as $n\to\infty$. In particular $\rho(u_1-u)<\infty$, so $u=u_1-(u_1-u)\in\Q$, and $\normq{u_n-u}\to0$ by Corollary \ref{norm-modular} \ref{nm1}. Therefore $u\in\x$ and $u_n\to u$ in $\x$. If the $u_n$ are radial, so is $u$, since it is the a.e.\ limit of $(u_n)$; thus $\X$ is closed.
\end{proof}


\subsection[Density of smooth compactly supported functions in X]{Density of $C_c^\infty(\R^N)$ in $\x$}

In this section we prove that the set of smooth functions with compact support is dense in $\x$. In other words we will prove that
\[
\x=\overline{C_c^\infty(\R^N)}^{\|\cdot\|},\quad\text{with }\|u\|=\|\nabla u\|_{L^2}+\|u\|_{Q_\gamma}.
\]
The proof is an adaptation of \cite[Proposition 2.6]{MR3568051}.

\begin{proposition}\label{den1}
    The space $C_c^\infty(\R^N)$ is dense in $\x$, and the space $C_{c,\text{rad}}^\infty(\R^N)$ of radial functions in $C_c^\infty(\R^N)$ is dense in $\X$.
\end{proposition}

\begin{proof}
    Let $\chi\in C^\infty(\R)$ be an even function with $0\leq\chi\leq1$, $\chi=0$ on $[-1,1]$ and $\chi=1$ on $\R\setminus[-2,2]$, and let $T(t)=\int_0^t\chi(\tau)d\tau$. (Note that a function with $|T'|\leq1$, $T=0$ on $[-1,1]$ and $T(t)=t$ for $|t|\geq2$, as is sometimes used, cannot exist.) For $n\in\N$, let $T_n(t)=T(nt)/n$, so that $0\leq T_n^\prime\leq1$, $|T_n(t)|\leq|t|$, $|T_n(t)-t|\leq2/n$, $T_n=0$ on $[-1/n,1/n]$, and $T_n^\prime(t)=\chi(nt)\to1$ for every $t\neq0$. Let $\psi\in C_c^{\infty}(\R^N)$ be a radial function with $\psi\geq0$, $\om\psi=1$ and $\text{supp}\,\psi\subset B_1$, and define
    \[
    \psi_n(x)=n^N\psi(nx).
    \]
    Let $u\in\x$ and set
    \[
    u_n:=T_n\circ(\psi_n\star u).
    \]
    Note that $u_n$ is well defined for all $n$ since $u\in L^{2^*}(\R^N)\subset L^1_{\text{loc}}(\R^N)$, that $u_n\in C^\infty(\R^N)$, and that $u_n$ is radial if $u$ is radial. Let us fix $\beta\in[a,b]$ with $\rho_\beta(u)<\infty$ (this holds for $\mu$-a.e.\ $\beta$). By Jensen's inequality (with the probability measure $\psi_n(x-y)dy$) and Proposition \ref{den},
    \begin{eqnarray}\label{den3}
        \left|\psi_n\star u(x)\right|^{q(\beta)}&\leq&\psi_n\star|u|^{q(\beta)}(x)\leq \|\psi\|_{L^\infty}n^N\int_{B_{1/n}(x)}|u|^{q(\beta)}\nonumber\\
        &\leq& Cn^{N-\frac{N-\beta}{2}}\seq{\int_{B_{1/n}(x)}\left|I_{\beta/2}\star|u|^{q(\beta)}\right|^2dy}^{1/2}.
    \end{eqnarray}
    Since $I_{\beta/2}\star|u|^{q(\beta)}\in L^2(\R^N)$, the right-hand side of \eqref{den3} tends to $0$ as $|x|\to\infty$, so that $\psi_n\star u(x)\to0$ as $|x|\to\infty$ and $\text{supp}\,u_n\subset\set{|\psi_n\star u|\geq1/n}$ is compact. Hence $u_n\in C^\infty_c(\R^N)$.

    \medskip\noindent{\it Convergence of the gradients.} Since $|T_n^\prime|\leq1$ and $\psi\geq0$, Jensen's inequality gives $|\nabla u_n|\leq|\psi_n\star\nabla u|\leq\psi_n\star|\nabla u|$, and therefore
    \[
    \int_{\R^N}|\nabla u_n|^2\leq\int_{\R^N}|\nabla u|^2.
    \]
    Since $\psi_n\star u\to u$ in $W^{1,1}_{\text{loc}}(\R^N)$ and $T_n^\prime(\psi_n\star u)\nabla(\psi_n\star u)\to\nabla u$ in $L^1_{\text{loc}}(\R^N)$ by the dominated convergence theorem (note that $\psi_n\star u\to u$ a.e., $T_n^\prime(t)\to1$ for $t\neq0$, and $\nabla u=0$ a.e.\ on $\set{u=0}$), the bounded sequence $(\nabla u_n)$ converges weakly to $\nabla u$ in $L^2(\R^N)$. Together with $\limsup\|\nabla u_n\|_{L^2}\leq\|\nabla u\|_{L^2}$ this gives
    \begin{eqnarray}\label{den-grad}
        \lim_{n\to\infty}\int_{\R^N}|\nabla u_n-\nabla u|^2dx=0.
    \end{eqnarray}

    \medskip\noindent{\it Convergence in $\Q$.} By Lebesgue's differentiation theorem and $|T_n(t)-t|\leq2/n$, we have $u_n\to u$ a.e.\ in $\R^N$. Let $\beta$ be as above and put $F_n:=I_{\beta/2}\star|u_n|^{q(\beta)}$ and $F:=I_{\beta/2}\star|u|^{q(\beta)}$. By $|T_n(t)|\leq|t|$ and Jensen's inequality,
    \[
    F_n\leq I_{\beta/2}\star|\psi_n\star u|^{q(\beta)}\leq I_{\beta/2}\star\seq{\psi_n\star|u|^{q(\beta)}}=\psi_n\star F,
    \]
    and hence
    \begin{eqnarray}\label{den4-1}
    \|F_n\|_{L^2}\leq\|F\|_{L^2}\qquad\text{for all }n\in\N.
    \end{eqnarray}
    On the other hand, Fatou's lemma gives $F\leq\liminf_{n}F_n$ pointwise, and hence $\liminf_n\int F_nF\geq\int F^2$. Therefore
    \[
    \limsup_{n\to\infty}\|F_n-F\|_{L^2}^2=\limsup_{n\to\infty}\seq{\|F_n\|_{L^2}^2-2\int F_nF+\|F\|_{L^2}^2}\leq0,
    \]
    that is, $F_n\to F$ in $L^2(\R^N)$. Now fix a subsequence along which $F_n\to F$ a.e. For every $x$ with $F(x)<\infty$ and $F_n(x)\to F(x)$, the functions $I_{\beta/2}(x-\cdot)|u_n-u|^{q(\beta)}$ tend to $0$ a.e., and are bounded by $2^{q(\beta)-1}I_{\beta/2}(x-\cdot)(|u_n|^{q(\beta)}+|u|^{q(\beta)})$, whose integrals $2^{q(\beta)-1}(F_n(x)+F(x))$ converge to $2^{q(\beta)}F(x)$. The generalized dominated convergence theorem gives $G_n(x):=I_{\beta/2}\star|u_n-u|^{q(\beta)}(x)\to0$. Since $0\leq G_n\leq2^{q(\beta)-1}(F_n+F)$ and $F_n+F\to2F$ in $L^2$, the generalized dominated convergence theorem once more gives $\|G_n\|_{L^2}\to0$. As every subsequence has a further subsequence with this property, we conclude that
    \begin{eqnarray}\label{den8}
    \lim_{n\to\infty}\rho_\beta(u_n-u)=0\qquad\text{for $\mu$-a.e.\ }\beta.
    \end{eqnarray}
    Moreover, by \eqref{den4-1},
    \begin{eqnarray}\label{den7}
        \rho_\beta(u_n-u)=\|G_n\|_{L^2}^2\leq2^{2q(\beta)-2}\seq{\|F_n\|_{L^2}+\|F\|_{L^2}}^2\leq2^{2q(b)}\rho_\beta(u),
    \end{eqnarray}
    which is $\mu$-integrable. By \eqref{den8}, \eqref{den7} and the dominated convergence theorem, $\rho(u_n-u)\to0$, i.e., $\normq{u_n-u}\to0$ by Corollary \ref{norm-modular} \ref{nm1}. Together with \eqref{den-grad} this shows that $u_n\to u$ in $\x$.
\end{proof}


\subsection{Uniform convexity and reflexivity}\label{uniform-convexity}

In this section we show that $\x$ is reflexive by proving that it is uniformly convex with respect to an equivalent norm. We stress that no restriction on $\gamma>0$ is needed, that is, it is enough that all the powers $q(\beta)$ are larger than one. When $\mu=\delta_\alpha$ is a Dirac mass, this is \cite[Proposition 2.8]{MR3568051}; the argument there, based on an isometric embedding of the Coulomb space into a vector-valued Lebesgue space, does not apply directly to the Luxemburg-type norm \eqref{normQ}, whose modular is not homogeneous. Instead, we prove a uniform convexity inequality for the modular $\rho$, based on the following elementary fact.

\begin{lemma}\label{pointwise-uc}
    Let $1<p_0\leq p_1<\infty$ and $\varepsilon\in(0,1]$. There exists $\delta=\delta(\varepsilon,p_0,p_1)\in(0,1)$ such that
    \[
    \left|\frac{s+t}{2}\right|^p\leq(1-\delta)\frac{|s|^p+|t|^p}{2}
    \]
    for all $p\in[p_0,p_1]$ and all $s,t\in\R$ with $|s-t|\geq\varepsilon\max\set{|s|,|t|}$.
\end{lemma}
\begin{proof}
    By homogeneity and symmetry we may assume that $\max\set{|s|,|t|}=|s|=1$, and replacing $(s,t)$ by $(-s,-t)$ if necessary, $s=1$. The set
    \[
    K=\set{(t,p):|t|\leq1,\ |1-t|\geq\varepsilon,\ p\in[p_0,p_1]}
    \]
    is compact and the function
    \[
    h(t,p)=\left|\frac{1+t}{2}\right|^p\Big/\frac{1+|t|^p}{2}
    \]
    is continuous on $K$. Since $p>1$, we have $|\frac{1+t}{2}|^p\leq(\frac{1+|t|}{2})^p\leq\frac{1+|t|^p}{2}$, with equality in the second inequality only if $|t|=1$ and in the first one only if $t\geq0$; hence equality holds only for $t=1$, which is excluded in $K$. Thus $h<1$ on $K$, and $\delta:=1-\max_Kh\in(0,1)$ has the required property.
\end{proof}

\begin{proposition}\label{modular-uc}
    Let $\varepsilon\in(0,1]$ and let $\delta=\delta(\varepsilon,q(a),q(b))$ be given by Lemma \ref{pointwise-uc}. Then for every $\beta\in[a,b]$ and all measurable $u,v$ with $\rho_\beta(u),\rho_\beta(v)<\infty$,
    \begin{equation}\label{muc-beta}
    \rho_\beta\seq{\frac{u+v}{2}}+\frac{\delta^2}{8}\rho_\beta\seq{\frac{u-v}{2}}\leq\seq{1+\delta^2\seq{\frac\varepsilon2}^{2q(a)}}\frac{\rho_\beta(u)+\rho_\beta(v)}{2}.
    \end{equation}
    Consequently, for every bounded Borel function $\omega:[a,b]\to[0,\infty)$ and all $u,v\in\Q$, the modular $\rho_\omega(u):=\int\omega(\beta)\rho_\beta(u)\,d\mu(\beta)$ satisfies
    \begin{equation}\label{muc}
    \rho_\omega\seq{\frac{u+v}{2}}+\frac{\delta^2}{8}\rho_\omega\seq{\frac{u-v}{2}}\leq\seq{1+\delta^2\seq{\frac\varepsilon2}^{2q(a)}}\frac{\rho_\omega(u)+\rho_\omega(v)}{2}.
    \end{equation}
\end{proposition}
\begin{proof}
    Fix $\beta\in[a,b]$, write $q=q(\beta)$, $\En=\En_\beta$, and set
    \[
    m=\left|\frac{u+v}{2}\right|^q,\quad \sigma=\frac{|u|^q+|v|^q}{2},\quad d=\sigma-m,\quad w=\left|\frac{u-v}{2}\right|^q,
    \]
    and
    \[
    A=\set{x:|u-v|\geq\varepsilon\max\set{|u|,|v|}}.
    \]
    By convexity $d\geq0$, and by Lemma \ref{pointwise-uc} (note that $\delta(\varepsilon,q(a),q(b))\leq\delta(\varepsilon,q,q)$) $d\geq\delta\sigma$ on $A$. By Lemma \ref{energy-props}, $\En(\sigma)\leq\frac{\rho_\beta(u)+\rho_\beta(v)}{2}<\infty$ and $\En(\sigma)=\En(m+d)\geq\En(m)+\En(d)$, so that
    \begin{equation}\label{muc1}
        \rho_\beta\seq{\frac{u+v}{2}}=\En(m)\leq\En(\sigma)-\En(d)\leq\En(\sigma)-\delta^2\En(\sigma\mathbf 1_A).
    \end{equation}
    On $A$ we have $w\leq\max\set{|u|,|v|}^q\leq2\sigma$, so $\En(w\mathbf 1_A)\leq4\En(\sigma\mathbf 1_A)$. On $\R^N\setminus A$ we have $|u-v|<\varepsilon\max\set{|u|,|v|}$, hence $w\leq(\varepsilon/2)^q2\sigma$ and $\En(w\mathbf 1_{\R^N\setminus A})\leq4(\varepsilon/2)^{2q}\En(\sigma)$. Using $\En(w)\leq2\En(w\mathbf 1_A)+2\En(w\mathbf 1_{\R^N\setminus A})$ we obtain
    \[
    \En(\sigma\mathbf 1_A)\geq\frac14\En(w\mathbf 1_A)\geq\frac18\En(w)-\frac14\En(w\mathbf1_{\R^N\setminus A})\geq\frac18\rho_\beta\seq{\frac{u-v}{2}}-\seq{\frac\varepsilon2}^{2q}\En(\sigma).
    \]
    Inserting this into \eqref{muc1}, and using $(\varepsilon/2)^{2q}\leq(\varepsilon/2)^{2q(a)}$ and $\En(\sigma)\leq\frac{\rho_\beta(u)+\rho_\beta(v)}{2}$, we obtain \eqref{muc-beta}. Inequality \eqref{muc} follows by multiplying \eqref{muc-beta} by $\omega(\beta)$ and integrating with respect to $\mu$, since $\rho_\beta(u),\rho_\beta(v)<\infty$ for $\mu$-a.e.\ $\beta$ when $u,v\in\Q$.
\end{proof}

\begin{theorem}\label{Q-uc}
    The space $Q_{\gamma}(\R^N)$ is uniformly convex.
\end{theorem}
\begin{proof}
    Let $\theta\in(0,2]$ and $u,v\in\Q$ with $\normq{u},\normq{v}\leq1$ and $\normq{u-v}\geq\theta$. Then $\rho(u),\rho(v)\leq1$ by Lemma \ref{ball property}, and $\normq{(u-v)/2}\geq\theta/2$, so that $\rho((u-v)/2)\geq(\theta/2)^{2q(b)}=:\eta_\theta$ by \eqref{two-sided}. Choose $\varepsilon\in(0,1]$ such that $(\varepsilon/2)^{2q(a)}\leq\eta_\theta/16$, and let $\delta=\delta(\varepsilon,q(a),q(b))$. By \eqref{muc} with $\omega\equiv1$,
    \[
    \rho\seq{\frac{u+v}{2}}\leq1+\frac{\delta^2\eta_\theta}{16}-\frac{\delta^2\eta_\theta}{8}=1-\frac{\delta^2\eta_\theta}{16},
    \]
    and Corollary \ref{norm-modular} \ref{nm3} gives $\normq{(u+v)/2}\leq(1-\delta^2\eta_\theta/16)^{1/(2q(b))}<1$, where the right-hand side depends only on $\theta$.
\end{proof}

We now consider the equivalent norm on $\x$ given by
\[
\|u\|_2:=\sqrt{\|u\|_{Q_\gamma}^2+\|\nabla u\|_{L^2}^2}.
\]

\begin{theorem}
    $(\x,\|\cdot\|_2)$ is a uniformly convex space.
\end{theorem}
\begin{proof}
    The map $T:\x\to \Q\times L^{2}(\R^N,\R^N)$, $T(u)=(u,\nabla u)$, is a linear isometry when the product is endowed with the norm $\|(f,g)\|=\sqrt{\|f\|^2_{Q_\gamma}+\|g\|^2_{L^{2}}}$. The product of two uniformly convex spaces endowed with this $\ell^2$-type norm is uniformly convex \cite[Theorem 3]{Day41}, and therefore so is its subspace $T(\x)$.
\end{proof}

\begin{corollary}
    $\x$ and $\X$ are reflexive Banach spaces.
\end{corollary}
\begin{proof}
    Uniformly convex Banach spaces are reflexive by the Milman--Pettis theorem, and reflexivity is preserved under equivalent renormings and by closed subspaces (Theorem \ref{banach}).
\end{proof}


\subsection{The superposed Hartree energy}

We define the {\sl superposed Hartree energy}
\[
\Hn(u):=\int_{(0,N)}\frac{\rho_\beta(u)}{2q(\beta)}\,d\mu(\beta)=\int_{(0,N)}\frac{1}{2q(\beta)}\int_{\R^N}\left|I_{\beta/2}\star|u|^{q(\beta)}\right|^2dx\,d\mu(\beta),\quad u\in\Q.
\]
Clearly $\frac{1}{2q(b)}\rho\leq\Hn\leq\frac{1}{2q(a)}\rho$, and $\Hn$ is convex.

\begin{lemma}\label{C1}
    $\Hn\in C^1(\Q,\R)$, and for all $u,v\in\Q$,
    \begin{align}\label{Hprime}
    \Hn^\prime(u)v&=\int_{(0,N)}\int_{\R^N}\seq{I_\beta\star|u|^{q(\beta)}}|u|^{q(\beta)-2}uv\,dx\,d\mu(\beta)\nonumber\\
    &=\int_{(0,N)}\int_{\R^N}\seq{I_{\beta/2}\star|u|^{q(\beta)}}\seq{I_{\beta/2}\star\seq{|u|^{q(\beta)-2}uv}}dx\,d\mu(\beta),
    \end{align}
    where the integrals converge absolutely, and
    \begin{equation}\label{Hprime-bound}
        \int_{(0,N)}\int_{\R^N}\seq{I_\beta\star|u|^{q(\beta)}}|u|^{q(\beta)-1}|v|\,dx\,d\mu(\beta)\leq\int_{(0,N)}\rho_\beta(u)^{1-\frac{1}{2q(\beta)}}\rho_\beta(v)^{\frac{1}{2q(\beta)}}d\mu(\beta)\leq\rho(u)+\rho(v).
    \end{equation}
    Moreover, $\Hn^\prime$ is monotone, and $\Hn$ is weakly lower semicontinuous on $\Q$ and on $\x$.
\end{lemma}
\begin{proof}
    Fix $\beta\in[a,b]$ and write $q=q(\beta)$. By Lemma \ref{energy-props} \ref{ep2}--\ref{ep3},
    \begin{align}\label{holder-riesz}
    \int_{\R^N}\seq{I_\beta\star|u|^q}|u|^{q-1}|v|\,dx&=\Bn_\beta\seq{|u|^q,|u|^{q-1}|v|}\leq\rho_\beta(u)^{\frac12}\seq{\rho_\beta(u)^{\frac{q-1}{q}}\rho_\beta(v)^{\frac1q}}^{\frac12}\nonumber\\
    &=\rho_\beta(u)^{1-\frac{1}{2q}}\rho_\beta(v)^{\frac{1}{2q}},
    \end{align}
    and \eqref{Hprime-bound} follows by Young's inequality and integration in $\beta$.

    {\it Step 1: Gateaux differentiability.} Let $u,v\in\Q$ and fix $\beta$ with $\rho_\beta(u),\rho_\beta(v)<\infty$. Put $h=|u|+|v|$; then $\rho_\beta(h)\leq2^{2q-1}(\rho_\beta(u)+\rho_\beta(v))<\infty$. For $|t|\leq1$, the derivative of $t\mapsto|u(x)+tv(x)|^q|u(y)+tv(y)|^q$ is bounded by $q\,h^{q-1}|v|(x)\,h^q(y)+q\,h^q(x)\,h^{q-1}|v|(y)$, which is integrable with respect to $I_\beta(x-y)dxdy$ by \eqref{holder-riesz} (with $h$ in place of $u$). Therefore $g_\beta(t):=\rho_\beta(u+tv)$ is differentiable, with $g_\beta^\prime(0)=2q\int(I_\beta\star|u|^q)|u|^{q-2}uv\,dx$. Since $g_\beta$ is convex, for $0<t\leq1$,
    \[
    g_\beta(0)-g_\beta(-1)\leq g_\beta^\prime(0)\leq\frac{g_\beta(t)-g_\beta(0)}{t}\leq g_\beta(1)-g_\beta(0),
    \]
    and the difference quotients decrease to $g^\prime_\beta(0)$ as $t\searrow0$. The bounds are $\mu$-integrable since $u\pm v\in\Q$, so the dominated convergence theorem shows that the right derivative of $t\mapsto\Hn(u+tv)$ at $0$ equals $\int\frac{1}{2q(\beta)}g^\prime_\beta(0)\,d\mu$; the left derivative is treated in the same way. This proves that $\Hn$ is Gateaux differentiable with derivative \eqref{Hprime} (the second expression follows from \eqref{s2}); by \eqref{Hprime-bound} and \eqref{two-sided}, $\Hn^\prime(u)\in\Q^*$.

    {\it Step 2: continuity of $\Hn^\prime$.} Let $w_n\to w$ in $\Q$. It suffices to show that every subsequence has a further subsequence along which $\sup_{\rho(v)\leq1}|(\Hn^\prime(w_n)-\Hn^\prime(w))v|\to0$ (recall that $\rho(v)\leq1$ iff $\normq v\leq1$). Since $\rho(w_n-w)\to0$, after passing to a subsequence we have $\rho_{\beta_0}(w_n-w)\to0$ for some $\beta_0$, hence $w_n\to w$ in $L^{q(\beta_0)}_{\text{loc}}(\R^N)$ by Proposition \ref{den}, and $w_n\to w$ a.e.\ along a further subsequence. We write
    \begin{align*}
    (\Hn^\prime(w_n)-\Hn^\prime(w))v&=\int\!\!\int\seq{I_\beta\star\seq{|w_n|^q-|w|^q}}|w_n|^{q-2}w_nv\\
    &\quad+\int\!\!\int\seq{I_\beta\star|w|^q}\seq{|w_n|^{q-2}w_n-|w|^{q-2}w}v=:T_1+T_2.
    \end{align*}
    Since $||w_n|^q-|w|^q|\leq q|w_n-w|(|w_n|+|w|)^{q-1}$, Lemma \ref{energy-props} gives
    \[
    \left|\int\seq{I_\beta\star\seq{|w_n|^q-|w|^q}}|w_n|^{q-2}w_nv\right|\leq q\,\rho_\beta(w_n-w)^{\frac{1}{2q}}\Big[\rho_\beta(|w_n|+|w|)\rho_\beta(w_n)\Big]^{\frac{q-1}{2q}}\rho_\beta(v)^{\frac{1}{2q}}.
    \]
    By Young's inequality in the form $x^{\frac{1}{2q}}y^{1-\frac{1}{2q}}\leq\varepsilon^{1-2q}x+\varepsilon y$ and the weighted arithmetic-geometric mean inequality, for $\varepsilon\in(0,1]$,
    \[
    \sup_{\rho(v)\leq1}|T_1|\leq q(b)\left[\varepsilon^{1-2q(b)}\rho(w_n-w)+\varepsilon\seq{\rho(|w_n|+|w|)+\rho(w_n)+1}\right],
    \]
    and since $(w_n)$ is bounded in $\Q$, $\limsup_n\sup_{\rho(v)\leq1}|T_1|\leq C\varepsilon$ for every $\varepsilon$. To estimate $T_2$, let $d\nu_\beta=(I_\beta\star|w|^q)dx$ and $Z_n(\beta)=\int||w_n|^{q-2}w_n-|w|^{q-2}w|^{\frac{q}{q-1}}d\nu_\beta$. By H\"older's inequality and Lemma \ref{energy-props} \ref{ep2},
    \[
    |T_2|\leq\int Z_n(\beta)^{\frac{q-1}{q}}\seq{\int|v|^qd\nu_\beta}^{\frac1q}d\mu\leq\int Z_n(\beta)^{\frac{q-1}{q}}\seq{\rho_\beta(w)\rho_\beta(v)}^{\frac{1}{2q}}d\mu
    \]
    \[
    \leq\varepsilon^{-\frac{1}{q(a)-1}}\int Z_n\,d\mu+\varepsilon\,\frac{\rho(w)+\rho(v)}{2}.
    \]
    The integrand of $\int Z_nd\mu=\iint||w_n|^{q-2}w_n-|w|^{q-2}w|^{\frac{q}{q-1}}d\nu_\beta\,d\mu$ tends to $0$ a.e.\ and is bounded by $C(|w_n-w|^{q}+|w|^{q})$, while $\iint|w_n-w|^qd\nu_\beta d\mu=\int\Bn_\beta(|w|^q,|w_n-w|^q)d\mu\leq\rho(w)^{1/2}\rho(w_n-w)^{1/2}\to0$ and $\iint|w|^qd\nu_\beta d\mu=\rho(w)<\infty$. By the generalized dominated convergence theorem, $\int Z_nd\mu\to0$, and hence $\limsup_n\sup_{\rho(v)\leq1}|T_2|\leq C\varepsilon$ for every $\varepsilon$. This proves that $\Hn^\prime:\Q\to\Q^*$ is continuous, and therefore $\Hn\in C^1(\Q,\R)$.

    {\it Step 3.} Since $\Hn$ is convex and $C^1$, $\Hn^\prime$ is monotone and $\Hn$ is weakly lower semicontinuous on $\Q$. As the inclusion $\x\hookrightarrow\Q$ is linear and continuous, weak convergence in $\x$ implies weak convergence in $\Q$, so $\Hn$ is weakly lower semicontinuous on $\x$ as well.
\end{proof}

\begin{remark}
    The same argument shows that $\rho\in C^1(\Q,\R)$. In particular, the functional $\Phi$ of the introduction is of class $C^1$ on $\x$, since $\x\hookrightarrow L^\ell(\R^N)$ for $\ell\in[q_\gamma,2^*]$ by Theorem \ref{emb} (when $q_\gamma\leq 2^*$).
\end{remark}


\subsection{Superposed Coulomb--Sobolev embeddings: proof of Theorem \ref{th0}}

Recall that $\alpha\in(0,N)$ has the positive neighborhood property with respect to $\mu$ if $\mu(I)>0$ for every open interval $I\subset(0,N)$ containing $\alpha$. We will use the following results of Mercuri, Moroz and Van Schaftingen \cite[Theorems 4 and 5, Proposition 6.6]{MR3568051}, stated in our notation.

\begin{theorem}\label{radial-compact-embedding}
    Let $N\geq3$, $\beta\in(1,N)$, $p\in[1,\infty)$ and $\ell\in[1,\infty)$, and let $q^\beta_{\text{rad}}(p)=\frac{2(2p(N-1)+N-\beta)}{3N+\beta-4}$. If
    \begin{enumroman}
        \item either $p<\dfrac{N+\beta}{N-2}$ and\, $q^\beta_{\text{rad}}(p)<\ell\leq2^*$,
        \item or  $p>\dfrac{N+\beta}{N-2}$ and\, $2^*\leq \ell<q^\beta_{\text{rad}}(p)$,
    \end{enumroman}
    then $E^{\beta,p}_{\text{rad}}(\R^N)\hookrightarrow L^\ell(\R^N)$ continuously. If moreover $\ell\neq 2^*$, the embedding is compact.
\end{theorem}

\begin{proof}[Proof of Theorem \ref{th0}]
We consider case (i), since case (ii) is analogous. If $r=2^*$, the embedding follows from the Sobolev inequality, so let $\qrad<r<2^*$. Since $q(\cdot)$ and $\beta\mapsto q^\beta_{\text{rad}}(q(\beta))$ are continuous, there is $\delta_0>0$ such that $[\alpha-\delta_0,\alpha+\delta_0]\subset(1,N)$ and, with $J:=[\alpha-\delta_0,\alpha+\delta_0]\cap[a,b]$,
\begin{eqnarray}\label{comp2}
   q(\beta)<\frac{N+\beta}{N-2}\quad\text{and}\quad q^\beta_{\text{rad}}(q(\beta))<r\qquad\text{for all }\beta\in J.
\end{eqnarray}
Since $\alpha\in\mathcal P_\mu$, $\mu(J)>0$.

{\it Continuity.} By Theorem \ref{radial-compact-embedding}, for every $\beta\in J$ there exists $C(\beta)$ such that $\|u\|_{L^r}\leq C(\beta)\|u\|_{E^{\beta,q(\beta)}}\leq C(\beta)G_\beta(u)$ for all $u\in E^{\beta,q(\beta)}_{\text{rad}}(\R^N)$. Lemma \ref{uniform-constants} (with $\mathcal V=C^\infty_{c,\text{rad}}(\R^N)$) gives $\|u\|_{L^r}\leq K\|u\|$ for all $u\in\X$.

{\it Compactness.} Let $\s{u_n}\subset\X$ be bounded. Then $\sup_n\|\nabla u_n\|_{L^2}\leq C$ and $\sup_n\rho(u_n)\leq C$ by Corollary \ref{norm-modular}. For $k\in\N$ let
\[
A_k:=\set{\beta\in[a,b]:\liminf_{n\to\infty}\rho_\beta(u_n)\geq k}.
\]
By Fatou's lemma, $k\mu(A_k)\leq\int\liminf_n\rho_\beta(u_n)d\mu\leq C$, so we can choose $k$ with $\mu(A_k)<\mu(J)$. Then there exists $\beta\in J\setminus A_k$, i.e., $\rho_\beta(u_n)<k$ for infinitely many $n$. Along this subsequence, $\s{u_n}$ is bounded in $E^{\beta,q(\beta)}_{\text{rad}}(\R^N)$, and by \eqref{comp2} and Theorem \ref{radial-compact-embedding} it has a further subsequence which converges in $L^r(\R^N)$. Hence the embedding $\X\hookrightarrow L^r(\R^N)$ is compact.
\end{proof}



\section{Proof of Theorems \ref{th1} and \ref{th2}}\label{sec4}
Throughout this section we assume \ref{A1}--\ref{A3} and fix $\alpha\in\mathcal P_\mu\cap(1,N)$. By Remark \ref{rem-qgamma} and Theorem \ref{th0},
\begin{equation}\label{embeddings}
q_{\rm rad}^\alpha<q_\gamma<2^*,\qquad\X\hookrightarrow L^\ell(\R^N)\ \begin{cases}\text{continuously for }\ell\in(q^\alpha_{\rm rad},2^*],\\ \text{compactly for }\ell\in(q^\alpha_{\rm rad},2^*).\end{cases}
\end{equation}
Following the notations of the abstract theory of Section \ref{sec2}, we consider
\[
W=\X,\quad\|u\|=\|\nabla u\|_{L^2}+\|u\|_{Q_\gamma},
\]
which is a reflexive Banach space by Section \ref{uniform-convexity}. On $W$ we use the scaling
\[
u_t(x)=t^\gamma u(tx),\qquad t\geq0,
\]
and we set
\[
s=2(\gamma+1)-N>0.
\]

\subsection{The scaling}

\begin{lemma}\label{scaling-identities}
    Let $u\in\x$, $t>0$ and $\beta\in(0,N)$. Then $u_t\in\x$, $u_t$ is radial if $u$ is, and
    \[
    \|\nabla u_t\|_{L^2}^2=t^s\|\nabla u\|_{L^2}^2,\qquad \rho_\beta(u_t)=t^s\rho_\beta(u),\qquad\int_{\R^N}|u_t|^\ell dx=t^{\gamma\ell-N}\int_{\R^N}|u|^\ell dx\quad(\ell\geq1).
    \]
    In particular $\gamma q_\gamma-N=s$, $\gamma2^*-N=\frac{N}{N-2}s$ and $\gamma r-N=s+\gamma(r-q_\gamma)$. Moreover,
    \begin{equation}\label{scaling-norm}
    \normq{u_t}\leq m(t)\normq{u},\qquad m(t):=\begin{cases}t^{s/(2q(b))}&\text{if }t\leq1,\\ t^{s/(2q(a))}&\text{if }t\geq1,\end{cases}
    \end{equation}
    and consequently $\|u_t\|\leq\max\set{t^{s/2},m(t)}\|u\|\leq\max\set{1,t^{s/2}}\|u\|$.
\end{lemma}
\begin{proof}
    The identities follow from the change of variables $x\mapsto tx$, from $(I_{\beta/2}\star f(t\,\cdot))(x)=t^{-\beta/2}(I_{\beta/2}\star f)(tx)$, and from $2\gamma q(\beta)=2\gamma+\beta+2$, which give $\rho_\beta(u_t)=t^{2\gamma q(\beta)-\beta-N}\rho_\beta(u)=t^s\rho_\beta(u)$. To prove \eqref{scaling-norm}, let $u\neq0$, $w=u/\normq{u}$ (so $\rho(w)=1$ by Lemma \ref{ball property}) and $\tau=m(t)\normq{u}$. Then $\rho(u_t/\tau)=t^s\rho(m(t)^{-1}w)$. If $t\geq1$, then $m(t)^{-1}\leq1$ and Lemma \ref{modular} gives $\rho(m(t)^{-1}w)\leq m(t)^{-2q(a)}=t^{-s}$; if $t\leq1$, then $m(t)^{-1}\geq1$ and $\rho(m(t)^{-1}w)\leq m(t)^{-2q(b)}=t^{-s}$. In both cases $\rho(u_t/\tau)\leq1$, i.e., $\normq{u_t}\leq\tau$. The last assertion follows since $\frac{s}{2q(b)}\leq\frac{s}{2q(a)}<\frac s2$.
\end{proof}

\begin{lemma}\label{scaling}
    The mapping $\X\times[0,\infty)\to\X$, $(u,t)\mapsto u_t$, is continuous and satisfies $(H_1)$--$(H_5)$, with $\kappa=s/2$ in $(H_5)$.
\end{lemma}
\begin{proof}
    $(H_1)$--$(H_3)$ are clear, and $(H_4)$, $(H_5)$ follow from Lemma \ref{scaling-identities}. For continuity, we first show that $\|u_\tau-u\|\to0$ as $\tau\to1$ for every $u\in\X$. If $\varphi\in C^\infty_{c,\rm rad}(\R^N)$, then for $\tau\in[\frac12,2]$ the functions $\varphi_\tau$ are supported in a fixed ball $B$, $\nabla\varphi_\tau\to\nabla\varphi$ in $L^2$ and $\|\varphi_\tau-\varphi\|_{L^\infty}\to0$; hence $\rho(\varphi_\tau-\varphi)\leq\int\|\varphi_\tau-\varphi\|_{L^\infty}^{2q(\beta)}\En_\beta(\mathbf 1_B)d\mu\to0$, since $\sup_{\beta\in[a,b]}\En_\beta(\mathbf1_B)<\infty$. For general $u\in\X$ and $\varphi\in C^\infty_{c,\rm rad}(\R^N)$, Lemma \ref{scaling-identities} gives
    \[
    \|u_\tau-u\|\leq\|(u-\varphi)_\tau\|+\|\varphi_\tau-\varphi\|+\|\varphi-u\|\leq\seq{\max\set{1,\tau^{s/2}}+1}\|u-\varphi\|+\|\varphi_\tau-\varphi\|,
    \]
    and the claim follows from Proposition \ref{den1}. Now let $u_j\to u$ in $\X$ and $t_j\to t$. If $t>0$, then by $(H_1)$ and Lemma \ref{scaling-identities},
    \begin{align*}
    \|(u_j)_{t_j}-u_t\|&\leq\|(u_j-u)_{t_j}\|+\|(u_{t_j/t}-u)_t\|\\
    &\leq\max\set{1,t_j^{s/2}}\|u_j-u\|+\max\set{1,t^{s/2}}\|u_{t_j/t}-u\|\to0.
    \end{align*}
    If $t=0$, then $\|(u_j)_{t_j}\|\leq\max\set{t_j^{s/2},m(t_j)}\|u_j\|\to0=\|u_0\|$.
\end{proof}

\subsection{The operators}

 The operators $A_s,B_s:\X\to\X^*$ are given by
 \begin{eqnarray*}
     A_s(u)v&=&\int_{\R^N}\nabla u\cdot\nabla vdx+\Hn^\prime(u)v\\
     &=&\int_{\R^N}\nabla u\cdot\nabla vdx+\int_{(0,N)}\int_{\R^N}\seq{I_{\beta}\star|u|^{q(\beta)}}|u|^{q(\beta)-2}uv\,dx\,d\mu
 \end{eqnarray*}
 (see \eqref{Hprime} for the equivalent expression $\int\int(I_{\beta/2}\star|u|^{q(\beta)})(I_{\beta/2}\star(|u|^{q(\beta)-2}uv))dxd\mu$) and
 \begin{eqnarray*}
     B_s(u)v=\int_{\R^N}|u|^{q_\gamma-2}uv\,dx.
 \end{eqnarray*}
 Their potentials are
 \[
 I_s(u)=\dfrac{1}{2}\int_{\R^N}|\nabla u|^2dx+\Hn(u)=\dfrac{1}{2}\int_{\R^N}|\nabla u|^2dx+\int_{(0,N)}\int_{\R^N}\dfrac{1}{2q(\beta)}\left|I_{\beta/2}\star|u|^{q(\beta)}\right|^2dx\,d\mu(\beta),
 \]
 \[
 J_s(u)=\dfrac{1}{q_\gamma}\int_{\R^N}|u|^{q_\gamma}dx,
 \qquad
 \M_s=\set{u\in\X: I_s(u)=1},
 \]
 \[
 \w{\Psi}(u)=\dfrac{1}{J_s(u)}=\frac{q_\gamma}{\int_{\R^N}|u|^{q_\gamma}dx}\quad(u\in\M_s).
 \]

\begin{lemma}\label{operators}
    $A_s,B_s\in\A_s$, and $(H_6)$, $(H_8)$, $(H_9)$, $(H_{10})$ and $(H_{11})$ hold.
\end{lemma}
\begin{proof}
    $I_s\in C^1(\X,\R)$ by Lemma \ref{C1}, and $J_s\in C^1(\X,\R)$ by \eqref{embeddings}; both are even, so $A_s$ and $B_s$ are odd potential operators. If $\|v\|\leq1$, then $\rho(v)\leq1$, and \eqref{Hprime-bound} gives $|A_s(u)v|\leq\|\nabla u\|_{L^2}+\rho(u)+1$, so $A_s$ is bounded on bounded sets by Corollary \ref{norm-modular}; the same holds for $B_s$ by \eqref{embeddings}. The scaling identities $A_s(u_t)v_t=t^sA_s(u)v$ and $B_s(u_t)v_t=t^sB_s(u)v$ follow from the same changes of variables as in Lemma \ref{scaling-identities}. Hence $A_s,B_s\in\A_s$.

    $(H_6)$ and $(H_8)$ hold since $A_s(u)u=\|\nabla u\|^2_{L^2}+\rho(u)$ and $B_s(u)u=\int|u|^{q_\gamma}$. If $u_j\weak u$ in $\X$, then $u_j\to u$ in $L^{q_\gamma}(\R^N)$ by \eqref{embeddings}, so $|u_j|^{q_\gamma-2}u_j\to|u|^{q_\gamma-2}u$ in $L^{q_\gamma/(q_\gamma-1)}(\R^N)$ and $B_s(u_j)\to B_s(u)$ in $\X^*$; this is $(H_9)$. By \eqref{two-sided},
    \[
    I_s(u)\geq\frac12\|\nabla u\|_{L^2}^2+\frac{1}{2q(b)}\min\set{\normq{u}^{2q(a)},\normq{u}^{2q(b)}}\to\infty\quad\text{as }\|u\|\to\infty,
    \]
    which is $(H_{10})$. Finally, for $u\neq0$ the function $t\mapsto I_s(tu)=\frac{t^2}{2}\|\nabla u\|_{L^2}^2+\int\frac{t^{2q(\beta)}}{2q(\beta)}\rho_\beta(u)d\mu$ is continuous and strictly increasing on $[0,\infty)$ (note that $\nabla u\neq0$), vanishes at $t=0$ and tends to $\infty$ as $t\to\infty$; hence $I_s(tu)=1$ has a unique positive solution, which is $(H_{11})$.
\end{proof}

 \begin{lemma}\label{H7}
     Let $\set{u_j}_{j\in\N}\subset\X$ be a sequence such that $u_j\rightharpoonup u$ weakly for some $u\in\X$ and $A_s(u_j)(u_j-u)\to0$. Then $u_j\to u$ strongly. In particular, $(H_7)$ holds.
 \end{lemma}
 \begin{proof}
     Since $A_s(u)\in\X^*$, we have $A_s(u)(u_j-u)\to0$, and therefore
     \[
     \|\nabla u_j-\nabla u\|^2_{L^2}+\seq{\Hn^\prime(u_j)-\Hn^\prime(u)}(u_j-u)=(A_s(u_j)-A_s(u))(u_j-u)\to0.
     \]
     Both terms on the left are nonnegative by the monotonicity of $\Hn^\prime$ (Lemma \ref{C1}), so $\nabla u_j\to\nabla u$ in $L^2$ and $\Hn^\prime(u_j)(u_j-u)\to0$. By convexity, $\Hn(u)\geq\Hn(u_j)+\Hn^\prime(u_j)(u-u_j)$, so that $\limsup_j\Hn(u_j)\leq\Hn(u)$; together with the weak lower semicontinuity of $\Hn$ (Lemma \ref{C1}), we get $\Hn(u_j)\to\Hn(u)$. Moreover, $(u_j+u)/2\weak u$, so $\liminf_j\Hn((u_j+u)/2)\geq\Hn(u)$.

     Let $\varepsilon\in(0,1]$ and $\delta=\delta(\varepsilon,q(a),q(b))$. Applying \eqref{muc} with $\omega(\beta)=\frac{1}{2q(\beta)}$ (so that $\rho_\omega=\Hn$) to $u_j$ and $u$, we obtain
     \begin{align*}
     \frac{\delta^2}{8}\limsup_{j\to\infty}\Hn\seq{\frac{u_j-u}{2}}&\leq\limsup_{j\to\infty}\left[\seq{1+\delta^2\seq{\frac\varepsilon2}^{2q(a)}}\frac{\Hn(u_j)+\Hn(u)}{2}-\Hn\seq{\frac{u_j+u}{2}}\right]\\
     &\leq\delta^2\seq{\frac\varepsilon2}^{2q(a)}\Hn(u).
     \end{align*}
     Hence $\limsup_j\Hn((u_j-u)/2)\leq8(\varepsilon/2)^{2q(a)}\Hn(u)$ for every $\varepsilon\in(0,1]$, i.e., $\Hn((u_j-u)/2)\to0$. Since $\rho\leq2q(b)\Hn$ and $\rho(u_j-u)\leq2^{2q(b)}\rho((u_j-u)/2)$ by Corollary \ref{norm-modular} \ref{nm4}, we obtain $\rho(u_j-u)\to0$, that is, $\normq{u_j-u}\to0$. Therefore $u_j\to u$ in $\X$.
 \end{proof}

\subsection{Symmetric criticality, regularity and Pohozaev identity}

In order to treat \eqref{superposition problem} and \eqref{e1} simultaneously, we fix $\lambda\in\R$, $\eta\geq0$, $r\in[q_\gamma,2^*]$ and $\iota\in\set{0,1}$, and we set
\[
g(t)=\lambda|t|^{q_\gamma-2}t+\eta|t|^{r-2}t+\iota|t|^{2^*-2}t,\qquad G(t)=\frac{\lambda}{q_\gamma}|t|^{q_\gamma}+\frac\eta r|t|^r+\frac{\iota}{2^*}|t|^{2^*}.
\]
We say that $u\in\x$ is a weak solution of
\begin{equation}\label{general-eq}
    -\Delta u+\int_{(0,N)}\seq{I_\beta\star|u|^{q(\beta)}}|u|^{q(\beta)-2}u\,d\mu(\beta)=g(u)\quad\text{in }\R^N
\end{equation}
if $\int\nabla u\cdot\nabla v\,dx+\Hn^\prime(u)v=\int g(u)v\,dx$ for all $v\in\x$. Equation \eqref{e1} corresponds to $\eta=\iota=0$ and equation \eqref{superposition problem} to $\iota=1$. The associated functional
\[
\Phi_g(u)=I_s(u)-\int_{\R^N}G(u)\,dx
\]
is of class $C^1$ on $\x$ by Lemma \ref{C1} and Theorem \ref{emb}, and its critical points are exactly the weak solutions of \eqref{general-eq}.

\begin{lemma}[Symmetric criticality]\label{symcrit}
    Every critical point of the restriction of $\Phi_g$ to $\X$ is a critical point of $\Phi_g$ on $\x$, i.e., a weak solution of \eqref{general-eq}. In particular, every eigenfunction of \eqref{3-1} (with $W=\X$) is a weak solution of \eqref{e1}.
\end{lemma}
\begin{proof}
    The orthogonal group $O(N)$ acts linearly and isometrically on $\x$ by $(h\cdot u)(x)=u(h^{-1}x)$; indeed $\|\nabla(h\cdot u)\|_{L^2}=\|\nabla u\|_{L^2}$ and $\rho_\beta(h\cdot u)=\rho_\beta(u)$ since the kernels $I_{\beta/2}$ are radial. The functional $\Phi_g$ is invariant under this action, and the set of fixed points is $\X$. The conclusion follows from the principle of symmetric criticality of Palais \cite{Palais}. For the second assertion, apply the first one with $\eta=\iota=0$.
\end{proof}

\begin{lemma}\label{regularity}
    Let $u\in\x$ be a weak solution of \eqref{general-eq}. Then $u\in W^{2,\sigma}_{\text{loc}}(\R^N)$ for every $\sigma\in[1,\infty)$; in particular $u\in C^{1,\vartheta}_{\text{loc}}(\R^N)$ for every $\vartheta\in(0,1)$.
\end{lemma}
\begin{proof}
    Let $R>0$. For $\beta\in[a,b]$ we split $I_\beta\star|u|^{q(\beta)}=N_\beta+F_\beta$, where $N_\beta=I_\beta\star(\mathbf 1_{B_{2R}}|u|^{q(\beta)})$. By Corollary \ref{far-field}, $0\leq F_\beta\leq C\rho_\beta(u)^{1/2}$ on $B_R$, with $C$ independent of $\beta$, and $\int\rho_\beta(u)^{1/2}d\mu\leq\mu((0,N))^{1/2}\rho(u)^{1/2}<\infty$. Moreover, since $|u|^{q(\beta)}\leq1+|u|^{q(b)}$ and, for $x\in B_R$, $N_\beta(x)\leq\int_{B_{2R}}I_\beta(x-y)(1+|u(y)|^{q(b)})dy$, Young's inequality gives
    \begin{equation}\label{near-field}
    \|N_\beta\|_{L^{p}(B_R)}\leq\|I_\beta\mathbf 1_{B_{3R}}\|_{L^{t}}\,\|1+|u|^{q(b)}\|_{L^{m}(B_{2R})},\quad 1+\frac1p=\frac1t+\frac1m,
    \end{equation}
    where $\sup_{\beta\in[a,b]}\|I_\beta\mathbf 1_{B_{3R}}\|_{L^t}<\infty$ for every $1\leq t<\frac{N}{N-a}$; indeed, $|z|^{\beta-N}\leq C_R|z|^{a-N}$ for $|z|\leq 3R$ and $\beta\in[a,b]$, with $C_R=\max\set{1,(3R)^{b-a}}$. Let
    \[
    H(x):=\int_{(0,N)}\seq{I_\beta\star|u|^{q(\beta)}}(x)|u(x)|^{q(\beta)-2}u(x)\,d\mu(\beta),
    \]
    so that
    \[
    |H|\leq\seq{1+|u|^{q(b)-1}}\int_{(0,N)}\seq{N_\beta+F_\beta}d\mu\quad\text{on }B_R.
    \]

    {\it Step 1: $H\in L^1_{\rm loc}(\R^N)$.} Since $u\in L^{2^*}(\R^N)$, we may take $t=1$ and $p=m=2^*/q(b)$ in \eqref{near-field}. By \ref{A2}, $\frac{q(b)}{2^*}+\frac{q(b)-1}{2^*}<1$, so H\"older's inequality and Fubini's theorem give $H\in L^1(B_R)$. Consequently, $-\Delta u=g(u)-H\in L^1_{\rm loc}(\R^N)$ in the sense of distributions.

    {\it Step 2: $u\in L^\sigma_{\rm loc}(\R^N)$ for all $\sigma<\infty$.} By Kato's inequality \cite{Kato72}, $-\Delta|u|\leq\text{sgn}(u)(-\Delta u)$ in $\mathcal D^\prime(\R^N)$. Since $\text{sgn}(u)H=\int(I_\beta\star|u|^{q(\beta)})|u|^{q(\beta)-1}d\mu\geq0$, we obtain
    \[
    -\Delta|u|\leq V|u|\quad\text{in }\mathcal D^\prime(\R^N),\qquad V:=\lambda^+|u|^{q_\gamma-2}+\eta|u|^{r-2}+\iota|u|^{2^*-2}.
    \]
    Since $2<q_\gamma\leq r\leq2^*$ and $u\in L^{2^*}(\R^N)$, we have $0\leq V\leq C(1+|u|^{2^*-2})\in L^{N/2}_{\rm loc}(\R^N)$. The Brezis--Kato estimate \cite{BrezisKato} (see also \cite[Appendix B]{Struwe}), which is based on testing with truncated powers of $|u|$ multiplied by cut-off functions and applies to nonnegative subsolutions, yields $u\in L^\sigma_{\rm loc}(\R^N)$ for every $\sigma<\infty$.

    {\it Step 3: conclusion.} Now we may take $p=\infty$, $t<\frac{N}{N-a}$ and $m=t^\prime$ in \eqref{near-field}, which shows that $N_\beta$ is bounded on $B_R$ uniformly in $\beta\in[a,b]$. Hence $|H|\leq C(1+|u|^{q(b)-1})$ on $B_R$, and therefore $g(u)-H\in L^\sigma_{\rm loc}(\R^N)$ for every $\sigma<\infty$. By the Calder\'on--Zygmund estimates \cite[Chapter 9]{GT}, $u\in W^{2,\sigma}_{\text{loc}}(\R^N)$ for all $\sigma<\infty$, and by Morrey's embedding $u\in C^{1,\vartheta}_{\rm loc}(\R^N)$ for all $\vartheta\in(0,1)$.
\end{proof}

\begin{lemma}[Pohozaev identity]\label{pohozaev}
Let $u\in\x$ be a weak solution of \eqref{general-eq}. Then
\begin{eqnarray*}
    \dfrac{N-2}{2}\int_{\R^N}|\nabla u|^2dx+\int_{(0,N)}\dfrac{N+\beta}{2q(\beta)}\int_{\R^N}\left|I_{\beta/2}\star|u|^{q(\beta)}\right|^2dx\,d\mu(\beta)=N\int_{\R^N}G(u)\,dx.
\end{eqnarray*}
\end{lemma}
\begin{proof}
    Let $\varphi\in C_c^{\infty}(\R^N)$ with $\varphi=1$ on $B_1$, and for $\delta>0$ let $\Psi_\delta(x):=\varphi(\delta x)x$ and
    \[
    v_\delta:=\Psi_\delta\cdot\nabla u.
    \]
    By Lemma \ref{regularity}, $v_\delta$ is bounded, compactly supported and belongs to $W^{1,2}(\R^N)$, so that $v_\delta\in\x$. Testing \eqref{general-eq} with $v_\delta$ gives
    \[
    \int_{\R^N}\nabla u\cdot\nabla v_\delta\,dx+\Hn^\prime(u)v_\delta=\int_{\R^N}g(u)v_\delta\,dx.
    \]
    As in \cite[Proposition 3.1]{MOROZ2013153}, integrating by parts and using the dominated convergence theorem (note that $G(u)\in L^1(\R^N)$ by Theorem \ref{emb}),
    \[
    \lim_{\delta\to0}\int_{\R^N}\nabla u\cdot\nabla v_\delta=-\dfrac{N-2}{2}\int_{\R^N}|\nabla u|^2,
    \]
    \[
    \lim_{\delta\to0}\int_{\R^N}g(u)v_\delta=\lim_{\delta\to0}\int_{\R^N}\Psi_\delta\cdot\nabla\big(G(u)\big)=-N\int_{\R^N}G(u).
    \]
    For the nonlocal term, fix $\beta\in[a,b]$ with $\rho_\beta(u)<\infty$, write $q=q(\beta)$ and $f=|u|^q\in C^1(\R^N)$, so that $\nabla f=q|u|^{q-2}u\nabla u$. Symmetrizing and integrating by parts on $\set{|x-y|>\epsilon}$ and letting $\epsilon\to0$, exactly as in \cite[Proposition 3.1]{MOROZ2013153}, we obtain
    \begin{align*}
    \int_{\R^N}\seq{I_\beta\star|u|^q}|u|^{q-2}uv_\delta\,dx&=\frac1q\iint I_\beta(x-y)f(y)\,\Psi_\delta(x)\cdot\nabla f(x)\,dx\,dy\\
    &=-\frac{1}{2q}\iint I_\beta(x-y)f(x)f(y)K_\delta(x,y)\,dx\,dy,
    \end{align*}
    where
    \[
    K_\delta(x,y)=\divg\Psi_\delta(x)+\divg\Psi_\delta(y)-(N-\beta)\frac{(x-y)\cdot(\Psi_\delta(x)-\Psi_\delta(y))}{|x-y|^2}.
    \]
    Since $|\divg\Psi_\delta|\leq C_\varphi$ and $|\Psi_\delta(x)-\Psi_\delta(y)|\leq C_\varphi|x-y|$ with $C_\varphi=N\|\varphi\|_{L^\infty}+\|\,|z|\nabla\varphi\|_{L^\infty}$ independent of $\delta$, we have $|K_\delta|\leq(N+2)C_\varphi$, and $K_\delta(x,y)\to2N-(N-\beta)=N+\beta$ as $\delta\to0$. Since $\int\frac{1}{2q(\beta)}\iint I_\beta(x-y)|u(x)|^{q(\beta)}|u(y)|^{q(\beta)}dxdyd\mu=\Hn(u)<\infty$, the dominated convergence theorem on $\R^N\times\R^N\times[a,b]$ gives
    \[
    \lim_{\delta\to0}\Hn^\prime(u)v_\delta=-\int_{(0,N)}\dfrac{N+\beta}{2q(\beta)}\int_{\R^N}\left|I_{\beta/2}\star|u|^{q(\beta)}\right|^2dx\,d\mu,
    \]
    and this concludes the proof.
\end{proof}

\begin{corollary}\label{energy-identity}
    Let $u\in\x$ be a weak solution of \eqref{general-eq}. Then
    \begin{equation}\label{Is-identity}
    I_s(u)=\frac{\lambda}{q_\gamma}\int_{\R^N}|u|^{q_\gamma}dx+\eta\,\frac{\gamma r-N}{rs}\int_{\R^N}|u|^rdx+\frac\iota2\int_{\R^N}|u|^{2^*}dx.
    \end{equation}
    In particular:
    \begin{enumroman}
        \item\label{ei1} condition $(H_{12})$ holds for the pair $(A_s,B_s)$ on $W=\X$;
        \item\label{ei2} if $\iota=1$, then
        \[
        \Phi_g(u)=\eta\,\frac{\gamma(r-q_\gamma)}{rs}\int_{\R^N}|u|^rdx+\frac1N\int_{\R^N}|u|^{2^*}dx\geq0.
        \]
    \end{enumroman}
\end{corollary}
\begin{proof}
    Testing \eqref{general-eq} with $u$ gives
    \[
    \dev{u}+\rho(u)=\lambda\int_{\R^N}|u|^{q_\gamma}dx+\eta\int_{\R^N}|u|^rdx+\iota\int_{\R^N}|u|^{2^*}dx.
    \]
    We multiply this identity by $\frac\gamma s$, the Pohozaev identity of Lemma \ref{pohozaev} by $-\frac1s$, and add. Since $s=2\gamma+2-N$ and $2\gamma q(\beta)=2\gamma+\beta+2$, the coefficients become
    \[
    \frac\gamma s-\frac{N-2}{2s}=\frac12,\qquad\frac\gamma s-\frac{N+\beta}{2q(\beta)s}=\frac{2\gamma q(\beta)-N-\beta}{2q(\beta)s}=\frac{1}{2q(\beta)},
    \]
    \[
    \frac\gamma s-\frac{N}{q_\gamma s}=\frac{\gamma q_\gamma-N}{q_\gamma s}=\frac{1}{q_\gamma},\qquad \frac\gamma s-\frac{N}{rs}=\frac{\gamma r-N}{rs},\qquad\frac\gamma s-\frac{N}{2^*s}=\frac{\gamma2^*-N}{2^*s}=\frac12,
    \]
    which gives \eqref{Is-identity}. If $u\in\X$ is an eigenfunction of \eqref{3-1}, then $u$ is a weak solution of \eqref{e1} by Lemma \ref{symcrit}, and \eqref{Is-identity} with $\eta=\iota=0$ reads $I_s(u)=\lambda J_s(u)$; this is \ref{ei1}. Finally, if $\iota=1$, then \eqref{Is-identity} gives
    \[
    \Phi_g(u)=\eta\seq{\frac{\gamma r-N}{rs}-\frac1r}\int|u|^r+\seq{\frac12-\frac1{2^*}}\int|u|^{2^*}=\eta\frac{\gamma(r-q_\gamma)}{rs}\int|u|^r+\frac1N\int|u|^{2^*},
    \]
    because $\gamma r-N-s=\gamma(r-q_\gamma)$. Note that the terms involving $\lambda$ cancel, so no sign condition on $\lambda$ is needed.
\end{proof}

\subsection{A local Palais--Smale condition}

From now on, $\Phi$ denotes the restriction to $\X$ of the functional associated with \eqref{superposition problem}, that is, of $\Phi_g$ with $\iota=1$:
\begin{eqnarray*}
    \Phi(u)=I_s(u)-\lambda J_s(u)
    -\dfrac{\eta}{r}\int_{\R^N}|u|^rdx-\dfrac{1}{2^*}\int_{\R^N}|u|^{2^*}dx,\qquad u\in\X.
\end{eqnarray*}
Let
\[
S=\inf_{u\in \D^{1,2}(\R^N)\backslash{\set{0}}}\dfrac{\dev{u}}{\seq{\cri{u}}^{2/2^*}}
\]
be the best Sobolev constant.

\begin{lemma}\label{ps-sup}
    Assume \ref{A1}--\ref{A3}, and let $\lambda\in\R$, $\eta\geq0$ and $r\in[q_\gamma,2^*)$. Then $\Phi$ satisfies the $(PS)_c$ condition on $\X$ for all $c<\frac{1}{N}S^{N/2}$.
\end{lemma}
\begin{proof}
    Let $\s{u_n}\subset\X$ be a $(PS)_c$ sequence for $\Phi$ with $c<\frac{1}{N}S^{N/2}$, that is, $\Phi(u_n)\to c$ and $\Phi^\prime(u_n)\to0$ in $\X^*$.

    {\it Step 1: $\s{u_n}$ is bounded.} Suppose that $\|u_n\|\to\infty$ for a renamed subsequence. Then $I_s(u_n)\to\infty$ by $(H_{10})$, and we set
    \[
    \w{t}_n=I_s(u_n)^{1/s}\to\infty,\qquad\w{u}_n=(u_n)_{1/\w t_n},
    \]
    so that $\w{u}_n\in\mathcal{M}_s$ and $u_n=(\w{u}_n)_{\w{t}_n}$. Since $\M_s$ is bounded, $\s{\w u_n}$ is bounded in $\X$, and $\|u_n\|=O(\w t_n^{s/2})$ by Lemma \ref{scaling-identities}. By Lemma \ref{scaling-identities}, $\Phi(u_n)=c+o(1)$ and $\Phi^\prime(u_n)u_n=o(1)\|u_n\|$ read, after division by $\w t_n^s$,
    \begin{eqnarray}\label{ps-sup3}
        1=\frac{\lambda}{q_\gamma}a_n+\frac{\eta}{r}b_n+\frac{1}{2^*}c_n+o(1),\qquad
        A_s(\w u_n)\w u_n=\lambda a_n+\eta b_n+c_n+o(1),
    \end{eqnarray}
    where
    \[
    a_n=\int_{\R^N}|\w u_n|^{q_\gamma}dx,\qquad b_n=\w t_n^{\gamma(r-q_\gamma)}\int_{\R^N}|\w u_n|^rdx,\qquad c_n=\w t_n^{\frac{2s}{N-2}}\int_{\R^N}|\w u_n|^{2^*}dx.
    \]
    The left-hand side of the second identity in \eqref{ps-sup3} is bounded, $(a_n)$ is bounded by \eqref{embeddings}, and $b_n,c_n\geq0$; hence $(b_n)$ and $(c_n)$ are bounded. In particular,
       \begin{eqnarray}\label{ps-sup5}
           \int_{\R^N}|\w{u}_n|^{2^*}dx=O\seq{\w{t}_n^{-\frac{2s}{N-2}}}\to0.
       \end{eqnarray}
    Fix $l\in(\qrad,q_\gamma)$. Since $\s{\w u_n}$ is bounded in $L^l(\R^N)$, the interpolation inequality and \eqref{ps-sup5} give $a_n\to0$. If $r>q_\gamma$, the interpolation inequality
        \[
        \int_{\R^N}|\w{u}_n|^{r}dx\leq\seq{\int_{\R^N}|\w{u}_n|^ldx}^{1-\theta}\seq{\int_{\R^N}|\w{u}_n|^{2^*}dx}^{\theta},\qquad\theta=\frac{r-l}{2^*-l},
        \]
    gives $b_n=O(\w t_n^{f(r)})$ with $f(r)=\gamma(r-q_\gamma)-\frac{2s}{N-2}\frac{r-l}{2^*-l}$. The function $f$ is affine, $f(2^*)=\gamma(2^*-q_\gamma)-\frac{2s}{N-2}=0$ and $f(q_\gamma)<0$, so $f(r)<0$ and $b_n\to0$; if $r=q_\gamma$, then $b_n=a_n\to0$. Now \eqref{ps-sup3} gives $c_n\to2^*$ and $A_s(\w u_n)\w u_n\to2^*$. On the other hand,
    \[
    A_s(\w u_n)\w u_n=\|\nabla\w u_n\|_{L^2}^2+\rho(\w u_n)\leq2q(b)I_s(\w u_n)=2q(b),
    \]
    so that $2^*\leq2q(b)$, which contradicts \ref{A2}. Hence $\s{u_n}$ is bounded.

    {\it Step 2: the weak limit is a critical point.} Since $\X$ is reflexive, a renamed subsequence of $\s{u_n}$ converges weakly to some $u\in\X$. By \eqref{embeddings} and the Rellich--Kondrachov theorem, $u_n\to u$ in $L^\ell(\R^N)$ for every $\ell\in(\qrad,2^*)$ (in particular for $\ell=q_\gamma$ and $\ell=r$), in $L^\ell_{\rm loc}(\R^N)$ for every $\ell<2^*$, and a.e.\ in $\R^N$, while $u_n\weak u$ in $L^{2^*}(\R^N)$ and $\nabla u_n\weak\nabla u$ in $L^2(\R^N)$. By Proposition \ref{weak riesz} in the Appendix, $\Hn^\prime(u_n)v\to\Hn^\prime(u)v$ for every $v\in\X$, and hence $\Phi^\prime(u)v=\lim\Phi^\prime(u_n)v=0$ for every $v\in\X$. By Lemma \ref{symcrit}, $u$ is a weak solution of \eqref{superposition problem}, and by Corollary \ref{energy-identity} \ref{ei2},
    \begin{equation}\label{Phi-u-positive}
        \Phi(u)\geq0.
    \end{equation}

    {\it Step 3: strong convergence.} Set $v_n=u_n-u$. By the weak convergence in the Hilbert space $L^2$ and the Brezis--Lieb lemma \cite{BrezisLieb},
      \begin{eqnarray}\label{ps-sup9}
          \dev{u_n}&=&\dev{u}+\dev{v_n}+o(1),\nonumber\\\int_{\R^N}|u_n|^{2^*}dx&=&\int_{\R^N}|u|^{2^*}dx+\int_{\R^N}|v_n|^{2^*}dx+o(1),
      \end{eqnarray}
    and by Fatou's lemma (applied inside the convolution and then in $x$ and $\beta$) $\liminf_n\rho(u_n)\geq\rho(u)$. Passing to a further subsequence, we may assume that $\dev{v_n}\to\ell_0$, $\int|v_n|^{2^*}\to L_0$ and $\rho(u_n)-\rho(u)\to d_0\geq0$. Since $\Phi^\prime(u_n)u_n\to0$ (the sequence being bounded) and $\Phi^\prime(u)u=0$, subtracting the two identities and using \eqref{ps-sup9} and the strong convergence in $L^{q_\gamma}$ and $L^r$ we get
    \[
    \ell_0+d_0=L_0\leq S^{-2^*/2}\ell_0^{2^*/2},
    \]
    by the Sobolev inequality. Hence either $\ell_0=0$ or $\ell_0\geq S^{N/2}$. Suppose that $\ell_0\geq S^{N/2}$. Since $q(\beta)<\frac{2^*}{2}$ on $[a,b]$ by \ref{A2}, the weight $\omega(\beta)=\frac{1}{2q(\beta)}-\frac{1}{2^*}$ is positive on $[a,b]$, and Fatou's lemma gives $\liminf_n\int\omega\rho_\beta(u_n)d\mu\geq\int\omega\rho_\beta(u)d\mu$. Therefore
    \begin{eqnarray*}
        c&=&\lim_{n\to\infty}\seq{\Phi(u_n)-\frac{1}{2^*}\Phi^\prime(u_n)u_n}\\
        &=&\lim_{n\to\infty}\left[\dfrac{1}{N}\dev{u_n}+\int_{(0,N)}\omega(\beta)\rho_\beta(u_n)d\mu\right.\\
        &&\left.-\lambda\seq{\dfrac{1}{q_\gamma}-\dfrac{1}{2^*}}\int_{\R^N}|u_n|^{q_\gamma}dx-\eta\seq{\dfrac{1}{r}-\dfrac{1}{2^*}}\om|u_n|^{r}dx\right]\\
        &\geq&\frac{\ell_0}{N}+\Phi(u)-\frac{1}{2^*}\Phi^\prime(u)u=\frac{\ell_0}{N}+\Phi(u)\geq\frac1NS^{N/2},
    \end{eqnarray*}
    by \eqref{ps-sup9} and \eqref{Phi-u-positive}, which contradicts $c<\frac1NS^{N/2}$. Therefore $\ell_0=0$, i.e., $\nabla v_n\to0$ in $L^2$, and hence $v_n\to0$ in $L^{2^*}(\R^N)$. Finally,
    \[
    A_s(u_n)(u_n-u)=\Phi^\prime(u_n)(u_n-u)+\lambda\int|u_n|^{q_\gamma-2}u_nv_n+\eta\int|u_n|^{r-2}u_nv_n+\int|u_n|^{2^*-2}u_nv_n\to0,
    \]
    since $(u_n-u)$ is bounded, $v_n\to0$ in $L^{q_\gamma}$, $L^r$ and $L^{2^*}$, and $(u_n)$ is bounded in these spaces. By Lemma \ref{H7}, $u_n\to u$ in $\X$.
 \end{proof}

\subsection{Proofs of the main results}

By Lemmas \ref{scaling}, \ref{operators}, \ref{H7} and Corollary \ref{energy-identity} \ref{ei1}, conditions $(H_1)$--$(H_{12})$ are satisfied, so Theorem \ref{Theorem 301} and Lemma \ref{index-below} apply to the eigenvalue problem \eqref{e1} in $W=\X$. Moreover, for $u\in\M_s$ and $t\geq0$, Lemma \ref{scaling-identities} gives
    \begin{equation}\label{212-1}
        \Phi(u_t)=t^s\seq{1-\frac{\lambda}{\w{\Psi}(u)}}-\frac{\eta\, t^{s+\gamma(r-q_\gamma)}}{r}\om|u|^r dx-\frac{t^{\frac{N}{N-2}s}}{2^*}\om|u|^{2^*}dx.
    \end{equation}
Since $\M_s$ is bounded in $\X$, it is bounded in $L^\ell(\R^N)$ for every $\ell\in(\qrad,2^*]$ by \eqref{embeddings}. Fix $l\in(\qrad,q_\gamma)$. If $\Lambda>0$ and $u\in\M_s$ satisfies $\w\Psi(u)<\Lambda$, i.e., $\int|u|^{q_\gamma}>q_\gamma/\Lambda$, then the interpolation inequalities between $L^l$ and $L^{r}$ and between $L^l$ and $L^{2^*}$ show that
\begin{equation}\label{lower-bounds}
\int_{\R^N}|u|^rdx\geq c_\Lambda>0,\qquad\int_{\R^N}|u|^{2^*}dx\geq c_\Lambda>0,
\end{equation}
where $c_\Lambda$ depends only on $\Lambda$ (and on the parameters of the problem, but not on $\eta$).

 \begin{proof}[Proof of Theorem \ref{th1}]
     We will employ Theorem \ref{m1} with $c^*=\frac1NS^{N/2}$, which is possible by Lemma \ref{ps-sup}. Let $\lambda\in\R$ and $m\geq1$. Since $\lambda_j\nearrow\infty$, we can choose $k\geq1$ such that $\lambda<\lambda_k$ and $\lambda_{k+m-1}<\lambda_{k+m}$. By Theorem \ref{Theorem 301} \ref{301-3}, the open symmetric set $\M_s\backslash\w{\Psi}_{\lambda_{k+m}}$ has index $k+m-1$, so it contains a compact symmetric subset $A_0$ with $i(A_0)=k+m-1$ (see the proof of Proposition 3.1 in Degiovanni and Lancelotti \cite{MR2371112}). Let $B_0=\w{\Psi}_{\lambda_{k}}$, which is closed and symmetric, with $i(\M_s\setminus B_0)\leq k-1$ by Lemma \ref{index-below}. Thus \eqref{m2} holds. Let $R>\rho>0$ (to be chosen) and let $X$, $A$, $B$ be as in Theorem \ref{m1}.

    {\it The set $B$.} If $u\in B_0$, then $\w\Psi(u)\geq\lambda_k$, so that $1-\lambda/\w\Psi(u)\geq1-\lambda^+/\lambda_k=:2c_0>0$. Since $\M_s$ is bounded in $L^r$ and $L^{2^*}$, \eqref{212-1} gives
    \[
    \Phi(u_\rho)\geq\rho^s\seq{2c_0-C\eta\rho^{\gamma(r-q_\gamma)}-C\rho^{\frac{2s}{N-2}}}\geq c_0\rho^s>0\quad\text{for all }u\in B_0,
    \]
    provided $\rho=\rho(\eta)>0$ is small enough. Hence $\inf_B\Phi>0$.

    {\it The sets $A$ and $X$.} If $u\in A_0$, then $\w\Psi(u)<\lambda_{k+m}$, so \eqref{lower-bounds} holds with $\Lambda=\lambda_{k+m}$. Moreover, $1-\lambda/\w\Psi(u)=1-\lambda J_s(u)\leq1+|\lambda|\sup_{\M_s}J_s=:c_1$. Hence, by \eqref{212-1},
   \begin{equation}\label{212-2}
        \Phi(u_t)\leq c_1t^{s}-\eta c_2t^{s+\gamma(r-q_\gamma)}-c_3t^{\frac{N}{N-2}s}\quad\text{for all }u\in A_0,\ t\geq0,
   \end{equation}
    with constants $c_1,c_2,c_3>0$ independent of $\eta$. Choose $R>\rho$ (independent of $\eta$) such that $c_1R^s-c_3R^{\frac{N}{N-2}s}\leq0$; then $\sup_A\Phi\leq0$. Moreover, since $r>q_\gamma$, \eqref{212-2} gives
    \[
    \sup_{u\in X}\Phi(u)\leq\max_{t \geq0}\set{c_1t^s-\eta c_2t^{s+\gamma(r-q_\gamma)}}=\frac{\w{C}}{\eta^{s/(\gamma(r-q_\gamma))}},
    \]
    for some constant $\w{C}>0$ independent of $\eta$. Let $\eta_m>0$ be such that $\w C\eta_m^{-s/(\gamma(r-q_\gamma))}=\frac1NS^{N/2}$. For $\eta>\eta_m$, all the inequalities in \eqref{m1-2} hold, and Theorem \ref{m1} gives $m$ distinct pairs of critical points of $\Phi$ on $\X$ at levels in $(0,\frac1NS^{N/2})$. They are nontrivial since $\Phi(0)=0$, and they are weak solutions of \eqref{superposition problem} by Lemma \ref{symcrit}.
 \end{proof}

 \begin{proof}[Proof of Theorem \ref{th2}]
     We use Theorem \ref{m1} with $c^*=\frac1NS^{N/2}$ and $\eta=0$. Let $c=c_\Lambda$ be the constant in \eqref{lower-bounds} with $\Lambda=\lambda_{k+m}$, and let
     \[
     C_0:=\max_{\tau\geq0}\set{\tau-\frac{c}{2^*}\tau^{\frac{N}{N-2}}}=\frac2N\seq{\frac{N}{N-2}\cdot\frac{c}{2^*}}^{-\frac{N-2}{2}},
     \]
     so that $\max_{t\geq0}\set{\theta t^s-\frac{c}{2^*}t^{\frac{N}{N-2}s}}=C_0\theta^{N/2}$ for every $\theta>0$ (substitute $\tau=\theta^{-(N-2)/2}t^s$). Choose $\delta_k>0$ such that
     \[
     \delta_k\leq\frac{\lambda_k}{2},\qquad\delta_k<\lambda_{k+m}-\lambda_k,\qquad C_0\seq{\frac{2\delta_k}{\lambda_k}}^{N/2}<\frac1NS^{N/2},
     \]
     and let $\lambda\in(\lambda_k-\delta_k,\lambda_k)$ and $\varepsilon=\delta_k$. Since $\lambda_{k+m-1}=\lambda_k<\lambda_k+\varepsilon<\lambda_{k+m}$, Theorem \ref{Theorem 301} \ref{301-3} gives
        \[
        i(\M_s\backslash\w{\Psi}_{\lambda_{k}+\varepsilon})=k+m-1,
        \]
        so that the open symmetric set $\M_s\backslash\w{\Psi}_{\lambda_{k}+\varepsilon}$ contains a compact symmetric subset $A_0$ with $i(A_0)=k+m-1$ \cite{MR2371112}. Let $B_0=\w{\Psi}_{\lambda_k}$; then $i(\M_s\setminus B_0)\leq k-1$ by Lemma \ref{index-below}, and \eqref{m2} holds. Let $R>\rho>0$ and let $X$, $A$, $B$ be as in Theorem \ref{m1}. For $u\in\M_s$ and $t\geq0$, \eqref{212-1} with $\eta=0$ reads
        \begin{equation}\label{210-1}
            \Phi(u_t)=t^s\seq{1-\frac{\lambda}{\w{\Psi}(u)}}-\frac{t^{\frac{N}{N-2}s}}{2^*}\om|u|^{2^*}dx.
        \end{equation}
        If $u\in B_0$, then $1-\lambda/\w\Psi(u)\geq1-\lambda/\lambda_k>0$, and since $\M_s$ is bounded in $L^{2^*}$, there exists $\rho>0$ small enough so that
        \[
        \inf_{u\in B}\Phi(u)>0.
        \]
        If $u\in A_0$, then $\w\Psi(u)<\lambda_k+\varepsilon<\lambda_{k+m}$, hence $\int|u|^{2^*}\geq c$ by \eqref{lower-bounds}, and, since $\lambda>0$,
        \[
        0<1-\frac{\lambda}{\w\Psi(u)}<1-\frac{\lambda}{\lambda_k+\varepsilon}=\frac{\lambda_k+\varepsilon-\lambda}{\lambda_k+\varepsilon}<\frac{2\delta_k}{\lambda_k}=:\theta_k.
        \]
        Therefore \eqref{210-1} gives
        \[
        \Phi(u_t)\leq\theta_kt^s-\frac{c}{2^*}t^{\frac{N}{N-2}s}\quad\text{for all }u\in A_0,\ t\geq0,
        \]
        which yields
        \[
        \sup_{u\in A}\Phi(u)\leq\theta_kR^s-\frac{c}{2^*}R^{\frac{N}{N-2}s}\leq0
        \]
        if $R$ is sufficiently large, and
        \[
        \sup_{u\in X}\Phi(u)\leq\sup_{t\ge0}\set{\theta_kt^s-\frac{c}{2^*}t^{\frac{N}{N-2}s}}=C_0\theta_k^{N/2}<\frac1NS^{N/2}.
        \]
        Hence \eqref{m1-2} holds. By Lemma \ref{ps-sup} and Theorem \ref{m1}, $\Phi$ has $m$ distinct pairs of nontrivial critical points on $\X$ at levels in $(0,\frac1NS^{N/2})$, which are weak solutions of \eqref{BN superposition equation} by Lemma \ref{symcrit}.
 \end{proof}

\section{Appendix: weak continuity of the superposed Hartree term}\label{appendix}

In this appendix we assume \ref{A1} and \ref{A2}. We denote by $L^2(dx\,d\mu)$ the space $L^2(\R^N\times[a,b],dx\otimes\mu)$.

\begin{proposition}\label{weak convergence riesz}
    Let $\s{u_n}\subset\x$ be a bounded sequence and let $u\in\x$ be such that $u_n\to u$ a.e.\ in $\R^N$ and in $L^\ell_{\rm loc}(\R^N)$ for every $\ell<2^*$. Then
    \[
    F_n(x,\beta):=\seq{I_{\beta/2}\star|u_n|^{q(\beta)}}(x)\weak F(x,\beta):=\seq{I_{\beta/2}\star|u|^{q(\beta)}}(x)\quad\text{weakly in }L^2(dx\,d\mu).
    \]
\end{proposition}
\begin{proof}
    Since $\|F_n\|_{L^2(dx\,d\mu)}^2=\rho(u_n)$ is bounded by Corollary \ref{norm-modular}, and the linear span of the functions $\phi(x)\psi(\beta)$, with $\phi\in C_c^\infty(\R^N)$ and $\psi$ bounded and Borel on $[a,b]$, is dense in $L^2(dx\,d\mu)$, it suffices to show that
    \[
    \int_{(0,N)}\psi(\beta)\int_{\R^N}\seq{F_n-F}\phi\,dx\,d\mu=\int_{(0,N)}\psi(\beta)\int_{\R^N}\seq{|u_n|^{q(\beta)}-|u|^{q(\beta)}}\seq{I_{\beta/2}\star\phi}dy\,d\mu\to0
    \]
    for such $\phi$ and $\psi$. Let $\text{supp}\,\phi\subset B_{R_0}$. There is $C_\phi>0$ such that, for all $\beta\in[a,b]$, $|I_{\beta/2}\star\phi|\leq C_\phi$ on $\R^N$ and $|I_{\beta/2}\star\phi(y)|\leq C_\phi|y|^{\beta/2-N}$ for $|y|\geq2R_0$. Let $R\geq2R_0$. By Lemma \ref{tail} with $\kappa=N-\beta/2$ (so that $\kappa-\frac{N-\beta}{2}=\frac N2$),
    \[
    \int_{\R^N\setminus B_R}\seq{|u_n|^{q(\beta)}+|u|^{q(\beta)}}|I_{\beta/2}\star\phi|\,dy\leq CR^{-N/2}\seq{\rho_\beta(u_n)^{1/2}+\rho_\beta(u)^{1/2}},
    \]
    and $\int(\rho_\beta(u_n)^{1/2}+\rho_\beta(u)^{1/2})d\mu\leq\mu((0,N))^{1/2}(\rho(u_n)^{1/2}+\rho(u)^{1/2})\leq C$. On the other hand, since $||u_n|^{q}-|u|^{q}|\leq q|u_n-u|(|u_n|+|u|)^{q-1}\leq q(b)|u_n-u|(1+|u_n|+|u|)^{q(b)-1}$ for $q=q(\beta)\in[q(a),q(b)]$,
    \[
    \sup_{\beta\in[a,b]}\int_{B_R}\left||u_n|^{q(\beta)}-|u|^{q(\beta)}\right|dy\leq q(b)\|u_n-u\|_{L^{q(b)}(B_R)}\big\|1+|u_n|+|u|\big\|_{L^{q(b)}(B_R)}^{q(b)-1}\to0,
    \]
    because $q(b)<2^*$. Therefore $\limsup_n|\int\psi\int(F_n-F)\phi\,dx\,d\mu|\leq C\|\psi\|_{L^\infty}R^{-N/2}$ for every $R\geq2R_0$, and the claim follows.
\end{proof}

\begin{proposition}\label{weak riesz}
    Let $\s{u_n}$ and $u$ be as in Proposition \ref{weak convergence riesz}. Then
    \begin{eqnarray*}
        \Hn^\prime(u_n)v=\int_{(0,N)}\int_{\R^N}\seq{I_{\beta/2}\star|u_n|^{q(\beta)}}\seq{I_{\beta/2}\star\seq{|u_n|^{q(\beta)-2}u_nv}}dx\,d\mu\to\Hn^\prime(u)v
    \end{eqnarray*}
    for every $v\in\x$. In particular, this holds whenever $u_n\weak u$ weakly in $\x$.
\end{proposition}
\begin{proof}
    Let $v\in\x$ and put $H_n(x,\beta)=I_{\beta/2}\star(|u_n|^{q(\beta)-2}u_nv)(x)$ and $H(x,\beta)=I_{\beta/2}\star(|u|^{q(\beta)-2}uv)(x)$. By Lemma \ref{energy-props} \ref{ep3}, $\|H\|^2_{L^2(dx\,d\mu)}\leq\int\rho_\beta(u)^{\frac{q-1}{q}}\rho_\beta(v)^{\frac1q}d\mu\leq\rho(u)+\rho(v)$, and similarly for $H_n$. Since $\Hn^\prime(u_n)v=\langle F_n,H_n\rangle_{L^2(dx\,d\mu)}$ and $\Hn^\prime(u)v=\langle F,H\rangle_{L^2(dx\,d\mu)}$, by Proposition \ref{weak convergence riesz} it suffices to show that $H_n\to H$ strongly in $L^2(dx\,d\mu)$.

    Let $k_n(x,\beta)=|u_n|^{q(\beta)-2}u_nv-|u|^{q(\beta)-2}uv$, so that $k_n\to0$ a.e.\ on $\R^N\times[a,b]$ and $|k_n|\leq(|u_n|^{q(\beta)-1}+|u|^{q(\beta)-1})|v|$. For $R\geq1$ let $\Omega_R=B_R\cap\set{|v|\leq R}$. By Lemma \ref{energy-props},
    \[
    \|H_n-H\|^2_{L^2(dx\,d\mu)}\leq\int\En_\beta(|k_n|)d\mu\leq2\int\En_\beta(|k_n|\mathbf 1_{\Omega_R})d\mu+2\int\En_\beta(|k_n|\mathbf 1_{\R^N\setminus\Omega_R})d\mu.
    \]
    {\it The tail term.} By Lemma \ref{energy-props} \ref{ep2}--\ref{ep3} and Young's inequality, for $\varepsilon\in(0,1]$ and $q=q(\beta)$,
    \begin{align*}
    \En_\beta(|k_n|\mathbf 1_{\R^N\setminus\Omega_R})&\leq2\seq{\rho_\beta(u_n)^{\frac{q-1}{q}}+\rho_\beta(u)^{\frac{q-1}{q}}}\En_\beta\seq{|v|^q\mathbf 1_{\R^N\setminus\Omega_R}}^{\frac1q}\\
    &\leq2\varepsilon\seq{\rho_\beta(u_n)+\rho_\beta(u)}+4\varepsilon^{1-q(b)}\En_\beta\seq{|v|^{q(\beta)}\mathbf 1_{\R^N\setminus\Omega_R}}.
    \end{align*}
    Since $\mathbf 1_{\R^N\setminus\Omega_R}\to0$ a.e.\ as $R\to\infty$ and $\En_\beta(|v|^{q(\beta)}\mathbf 1_{\R^N\setminus\Omega_R})\leq\rho_\beta(v)$, the dominated convergence theorem gives $\int\En_\beta(|v|^{q(\beta)}\mathbf1_{\R^N\setminus\Omega_R})d\mu\to0$ as $R\to\infty$. Hence the tail term is bounded by $C\varepsilon+4\varepsilon^{1-q(b)}o_R(1)$, uniformly in $n$.

    {\it The local term.} Let $p_a=\frac{2N}{N+a}<2$. By the Hardy--Littlewood--Sobolev inequality \cite[Theorem 4.3]{LiebLoss}, whose constant is bounded for $\beta\in[a,b]$, and by H\"older's inequality on $B_R$, there exists $C_R>0$ such that $\En_\beta(f)\leq C\|f\|^2_{L^{2N/(N+\beta)}}\leq C_R\|f\|_{L^{p_a}}^2$ for every $f$ supported in $B_R$ and every $\beta\in[a,b]$. Now $|k_n|\mathbf 1_{\Omega_R}\leq g_n:=R(2+|u_n|^{q(b)-1}+|u|^{q(b)-1})\mathbf 1_{B_R}$, and $g_n\to g:=2R(1+|u|^{q(b)-1})\mathbf 1_{B_R}$ in $L^{p_a}(\R^N)$, because $u_n\to u$ in $L^\ell(B_R)$ for every $\ell<2^*$ and $(q(b)-1)p_a<\frac{2}{N-2}\cdot2<2^*$ by \ref{A2}. By the generalized dominated convergence theorem on $B_R\times[a,b]$, $\int\|k_n\mathbf 1_{\Omega_R}\|^{p_a}_{L^{p_a}}d\mu\to0$. Since $\|k_n\mathbf 1_{\Omega_R}\|_{L^{p_a}}\leq\|g_n\|_{L^{p_a}}\leq C$ and $p_a<2$,
    \[
    \int\En_\beta(|k_n|\mathbf 1_{\Omega_R})d\mu\leq C_R\,C^{2-p_a}\int\|k_n\mathbf 1_{\Omega_R}\|_{L^{p_a}}^{p_a}d\mu\to0\quad\text{as }n\to\infty.
    \]
    Letting first $n\to\infty$, then $R\to\infty$ and finally $\varepsilon\to0$, we obtain $\|H_n-H\|_{L^2(dx\,d\mu)}\to0$.

    Finally, if $u_n\weak u$ in $\x$, then $(u_n)$ is bounded and $u_n\weak u$ in $\mathcal D^{1,2}(\R^N)$; by the Rellich--Kondrachov theorem every subsequence has a further subsequence which converges to $u$ in $L^\ell_{\rm loc}(\R^N)$ for all $\ell<2^*$ and a.e., and the first part applies to it. Since the limit $\Hn^\prime(u)v$ does not depend on the subsequence, the whole sequence converges.
\end{proof}

\end{document}